\documentclass[11pt,letterpaper]{article}
\usepackage[T1]{fontenc}
\usepackage{lmodern}
\usepackage{amsmath,amssymb,amsthm,mathtools,bm,mathrsfs}
\usepackage{microtype}
\usepackage{geometry}
\usepackage{graphicx}
\usepackage{booktabs}
\usepackage{placeins}
\usepackage{enumitem}
\usepackage{cite}
\usepackage{float}
\usepackage[hidelinks]{hyperref}
\hypersetup{
  pdftitle={Higher-order Approximate Symmetries},
  pdfauthor={Alexey Shevyakov}}
\providecommand{\doi}[1]{doi:\,\href{https://doi.org/#1}{\nolinkurl{#1}}}
\def\beq{\begin{equation}}
\def\eeq{\end{equation}}
\def\barr{\begin{array}{ll}}
\def\earr{\end{array}}

\newcommand{\eps}{\epsilon}
\newcommand{\Boxop}{\Box}
\newcommand{\pd}[2]{\frac{\partial #1}{\partial #2}}
\newcommand{\DD}{\mathrm D}

\newcommand{\cE}{\mathcal E}

\newcommand{\ellop}{\ell}
\newcommand{\Eop}{\mathcal L}
\newcommand{\Qpoly}{\mathcal Q}
\newcommand{\Rpoly}{\mathcal R}
\newcommand{\const}{\mathrm{const}}

\newtheorem{theorem}{Theorem}[section]
\newtheorem{lemma}{Lemma}[section]
\newtheorem{proposition}{Proposition}[section]
\newtheorem{corollary}{Corollary}[section]
\numberwithin{equation}{section}

\theoremstyle{definition}
\newtheorem{definition}{Definition}[section]

\newtheorem{remark}{Remark}[section]

\title{Higher-order Approximate Symmetries}
\author{Alexey Shevyakov\thanks{The author's name is also transliterated as
Alexei Cheviakov; both spellings refer to the same person.
Correspondence: \texttt{shevyakov@math.usask.ca}.}\\[0.5ex]
\small Department of Mathematics and Statistics, University of Saskatchewan, Saskatoon, Canada}
\date{}

\begin{document}
\maketitle

\begin{abstract}
Fushchych and Shtelen defined approximate symmetry by expanding the solution of
a perturbed differential equation in the small parameter and replacing the
original equation by a triangular system for the expansion coefficients.
This construction is extended to arbitrary order. A coefficient-map
formulation gives the order-$p$ Fushchych--Shtelen (FS) system and establishes,
under explicit hypotheses, a correspondence with the
Baikov--Gazizov--Ibragimov (BGI) method.

The two frameworks are compared for a perturbed cubic wave equation.
The nonlinearities admitting an FS point symmetry with prescribed projection
onto the unperturbed dilation are classified through second order, up to
equivalences. The classification consists of a power family and logarithmic
families at exceptional exponents. The nonlinearities admitting second-order
BGI point symmetries that continue the unperturbed dilation are also
classified. Among perturbations nontrivial at first order, they form a proper
subfamily of one FS family; another family is admitted by both methods when
the perturbation first enters at second order. The two methods are related
but not interchangeable. Joint reduction by the Lorentz algebra and the
inherited FS dilation gives explicit invariant solutions for all resulting
FS normal forms. The solutions of the 1989 Fushchych--Shtelen letter are
recovered at first order. For the common family with a nontrivial first-order
perturbation, the logarithmic corrections are the expansion of a power law
with a perturbation-dependent exponent.

Periodic traveling waves are used to examine the long-scale validity of
the FS expansion. The first correction is obtained by quadrature, all
second-order secular terms are identified, and phase renormalization gives
bounded approximations with the wavenumber truncated at the corresponding
order. For integer-power members of the power family, the amplitude
dependence of the corrected wavenumber agrees with the scaling weights
obtained from the symmetry classification.
\end{abstract}

\noindent\textbf{Keywords:} approximate symmetries; Fushchych--Shtelen expansion; Baikov--Gazizov--Ibragimov method; nonlinear wave equation; scale invariance; perturbation theory.

\medskip
\noindent\textbf{2020 Mathematics Subject Classification:} 35B06, 35B20, 35C06.

\section{Introduction}\label{sec:intro}

Wilhelm I. Fushchych (1936--1997) founded the Ukrainian school of
group analysis of differential equations and consistently connected symmetry
classification with reductions and explicit solutions. A representative
monograph is \cite{FSS1993}, and an account of his scientific legacy is given in
\cite{Nikitin1997}. The present paper concerns a short 1989 letter written with
W.~M. Shtelen \cite{Fushchich1989}.

In that letter, Fushchych and Shtelen considered a nonlinear wave equation with
a small parameter. The solution was expanded in the parameter, and the
coefficients of equal powers were equated. This procedure replaces the
perturbed equation by an exact system for the successive coefficients of the
solution. Ordinary Lie symmetry analysis can then be applied to that system.
The letter concluded with approximate solutions. The first correction was
taken to be a function of the leading field, and the resulting equation closed
when the leading field satisfied a d'Alembert--Hamilton differential
constraint. As shown below, the explicit Lorentz-radial family in
\cite{Fushchich1989} follows from a joint Lorentz--dilation reduction. For this
invariant family, the differential constraint is a consequence of the
reduction. The same reduction extends to the second perturbation order;
it does not describe every solution of the d'Alembert--Hamilton system.

A differential equation containing a small parameter may have a larger
symmetry algebra in the limiting problem than at nonzero parameter values.
Approximate symmetry methods seek transformations that preserve the equation
to a prescribed perturbation order. In the Baikov--Gazizov--Ibragimov (BGI)
approach, the infinitesimal generator is expanded in the small parameter
\cite{Baikov1988,Baikov1989,Baikov1991,Ibragimov1995}. In the
Fushchych--Shtelen (FS) approach, the solution itself is expanded, and exact
Lie symmetry analysis is applied to the coefficient system
\cite{Fushchich1989}. Other systematic formulations have also been developed.
Pakdemirli, Y\"ur\"usoy, and Dolap\c{c}{\i} \cite{Pakdemirli2004} compared the
two constructions on a set of examples and proposed a modified
dependent-variable approach intended to reduce the computational cost of the
FS procedure. Di Salvo, Gorgone, and Oliveri \cite{DiSalvo2018} developed a
perturbation-consistent approximate Lie-group framework in which both the
dependent variables and the generator are expanded. These alternatives use
definitions different from the coefficient-system correspondence studied here.

Fushchych and Shtelen introduced their construction for the nonlinear wave
equation
\begin{equation}\label{eq:introFS}
\Boxop U+\lambda U^3+\eps F(U)=0
\end{equation}
in $(1+3)$-dimensional Minkowski space. They determined all nonlinearities $F$
for which the equation is approximately scale invariant and all those for
which it is approximately conformally invariant, and used these symmetries to
construct approximate solutions. Their final paragraph observed that the
expansion
\[
U=U_0+\eps U_1+\eps^2U_2+\cdots
\]
may be continued to higher orders, but the corresponding higher-order
classification was not carried out.

Comparative studies later showed that the BGI and FS methods need not classify
the same symmetries as stable under perturbation
\cite{Tarayrah2021,Tarayrah2023}. Here, stability means that a symmetry of the
limiting equation can be continued to an approximate symmetry of the
perturbed equation. The first-order relation was established in
\cite{Druzhkov2024}: every BGI approximate local symmetry has an FS counterpart,
whereas an FS symmetry need not arise from a BGI characteristic. The present
paper has four main goals.

First, the FS construction and the BGI--FS correspondence are formulated at an
arbitrary perturbation order. The order-$p$ FS system is triangular in the
following precise sense: for $j\geq1$, its $j$th equation is linear in the
newly introduced coefficient $U_j$, and the operator acting on $U_j$ is the
linearization of the unperturbed equation. A coefficient map records the
coefficient of each power of the small parameter. This map is used to prove
the all-orders correspondence theorem. Its converse requires local regularity
of the unperturbed equation and the FS-covering property defined in
Section~\ref{sec:coefficientmap}.

Second, the full first-order scale-invariance result of Fushchych and Shtelen,
including its two logarithmic cases, is recovered and extended to the
second-order FS system for
\begin{equation}\label{eq:introSecond}
\Boxop U+\lambda U^3+\eps F(U)+\eps^2G(U)=0.
\end{equation}
The symmetry is required to project to the dilation of the unperturbed cubic
wave equation. The cases are distinguished by whether the first term that
cannot be removed by equivalence occurs at order $\eps$ or at order $\eps^2$.
The general case gives a family of powers. At three exceptional exponents,
the forcing exponent in an Euler determining equation coincides with its
homogeneous exponent. When the corresponding forcing coefficient is nonzero,
the solution contains a logarithm. This coincidence of exponents is a
resonance of the determining equation. A second resonance of the same type
occurs in the reduced invariant-surface equations of
Section~\ref{sec:waves}; the relation between them is stated in
Proposition~\ref{prop:solres}. An independent BGI calculation classifies all
second-order point symmetries that continue the unperturbed dilation.
Allowing corrections to the independent-variable components gives no
additional pairs $(F,G)$, although it permits further symmetry generators
subject to compatibility conditions. The admissible pairs are compared with
the FS classification. The logarithmic BGI family occurs at an exceptional
exponent excluded from the general power formula and is obtained by solving
that case separately.

Third, the classification is used to construct solutions. The second-order FS
system is reduced jointly by the Lorentz algebra and the inherited FS
dilation. The leading equation becomes algebraic, the two remaining
invariance equations are first-order linear equations, and every classified
family is integrated in closed form. The exceptional exponents in the
constitutive classification are also the exceptional exponents of the reduced
solution equations. The first-order formulas recover the solutions in
\cite{Fushchich1989}; retaining the removable affine terms recovers the full
table in that paper. For the family with a nontrivial first-order perturbation
admitted by both methods, the second-order expression is the expansion of a
power whose exponent changes with the perturbation parameter.

Fourth, the long-scale validity of FS expansions is examined. A formal
coefficient hierarchy does not by itself guarantee accuracy when time or
distance grows with the inverse perturbation parameter. The first-order FS
wave solution in \cite{Tarayrah2023} provides a concrete secular-growth example.
For a related family of nonlinear wave equations, slow variables or an
equivalent renormalization of the leading profile remove such growth
\cite{McAdam2025}. For the periodic traveling waves of the perturbed cubic
equation, the first correction is obtained by quadrature and all second-order
phase terms are identified in Section~\ref{sec:cnoidal}. The resulting first-
and second-order phase-renormalized approximations are compared with numerical
solutions.

Second-order and higher-order FS calculations have been carried out for other
classes of perturbed nonlinear wave equations; see, for example,
\cite{Zhang2011}. The classification considered here is complete within the
class of point symmetries projectable onto the $(x,U_0)$-space with prescribed
projection equal to the unperturbed dilation, modulo the formal equivalences
of Section~\ref{sec:wave}. Such a second-order FS classification for
\eqref{eq:introSecond}, together with the complete classification of BGI point
symmetries continuing the same dilation, their comparison, the explicit
reduction, and the analysis of the exceptional exponents, has not, to the
author's knowledge, appeared in the literature.

The notation and the two approximate-symmetry constructions are introduced in
Section~\ref{sec:framework}. The order-$p$ FS system and the correspondence
theorem are given in Section~\ref{sec:higherorder}. The first- and second-order
classifications for \eqref{eq:introSecond} and the BGI comparison are presented
in Section~\ref{sec:wave}. The Lorentz--dilation invariant approximate
solutions and their ordering conditions are given in Section~\ref{sec:waves}.
Cnoidal traveling waves, secular growth, phase renormalization, and numerical
comparisons are considered in Section~\ref{sec:cnoidal}. The results and their
scope are discussed in Section~\ref{sec:discussion}.

\section{Exact and approximate local symmetries}\label{sec:framework}

Standard background on Lie group analysis of differential equations may be found in \cite{Ovsiannikov1982,Olver1993,Bluman2010}; for approximate symmetry in the BGI sense see \cite{Baikov1988,Baikov1989,Baikov1991,Ibragimov1995}.

Let $x=(x^1,\ldots,x^n)$ be independent variables and $u=(u^1,\ldots,u^m)$ dependent variables. The notation $u^a_\alpha=\DD_\alpha u^a$ is used for jet variables, where $\alpha$ is a multi-index and $\DD_\alpha$ is a composition of total derivatives. A differential function depends smoothly on $x$, $u$, and finitely many derivatives of $u$.

An evolutionary vector field with characteristic $\varphi=(\varphi^1,\ldots,\varphi^m)$ is
\begin{equation}\label{eq:evofield}
\widehat X_\varphi=\sum_{a,\alpha}\DD_\alpha(\varphi^a)\pd{}{u^a_\alpha}.
\end{equation}
For a point field $X=\xi^i(x,u)\partial_{x^i}+\eta^a(x,u)\partial_{u^a}$, the corresponding evolutionary characteristic is given by $\varphi^a=\eta^a-\xi^iu^a_i$.

\medskip Consider a system of differential equations $F_0[u]=0$, with $F_0=(F_0^1,\ldots,F_0^r)$. The following local regularity condition is used in the
first-order BGI--FS correspondence result \cite{Druzhkov2024}; it is the
finite-order, local version of the regularity (``normality'') assumptions
under which every differential function vanishing on the prolonged equation
belongs to the differential ideal generated by it, see \cite[Ch.~2]{Olver1993}.

\begin{definition}\label{def:regular}
The infinite prolongation of $F_0=0$ is \emph{regular} if, locally on a
sufficiently high finite jet space, every finite-order differential function
that vanishes on it can be represented as $\Delta(F_0)$ for a finite-order linear operator $\Delta$ in total derivatives, with
differential-function coefficients (a matrix operator for a system).
\end{definition}

Thus, for a regular equation, $\widehat X_\varphi$ is a local symmetry precisely when
\begin{equation}\label{eq:exactoperator}
\widehat X_\varphi(F_0)=\Delta(F_0)
\end{equation}
for an operator in total derivatives $\Delta$.

Let $\eps$ be a small parameter, and consider a perturbed family of differential equations
\begin{equation}\label{eq:perturbedfamily}
F[u;\eps]=\sum_{j=0}^{p}\eps^jF_j[u]+o(\eps^p).
\end{equation}
Throughout, the $\eps$-expansions may be read as formal truncated power
series; if actual parameter-dependent differential functions are used, they
are assumed sufficiently smooth in $\eps$ at $\eps=0$ for all coefficient
maps through order $p$ below to exist, with the displayed Taylor remainders.
The notation $A\approx_p B$ means that $A-B=o(\eps^p)$; thus $A\approx_pB$ means that
all coefficients of $A$ and $B$ up to the order $\eps^p$ agree. The little-$o$
convention is used because it requires only the coefficient maps through order
$p$ to exist and does not assume a $(p+1)$-st derivative in $\eps$; when a full
Taylor expansion is available, the remainder may equivalently be written
$O(\eps^{p+1})$, which is the convention often associated with
approximation to order $p$. An approximate characteristic has the form
\begin{equation}\label{eq:approxchar}
\varphi[u;\eps]\approx_p\sum_{j=0}^{p}\eps^j\varphi_j[u].
\end{equation}

\begin{definition}[BGI symmetry of order $p$]\label{def:BGI}
The approximate evolutionary field $\widehat X_\varphi$ is a BGI approximate symmetry of \eqref{eq:perturbedfamily} to order $p$ if there is an approximate operator in total derivatives
\[
\Delta\approx_p\Delta_0+\eps\Delta_1+\cdots+\eps^p\Delta_p
\]
such that
\begin{equation}\label{eq:BGIoperator}
\widehat X_\varphi(F)\approx_p\Delta(F).
\end{equation}
\end{definition}

\begin{remark}[Relation to the on-solutions formulation]\label{rem:BGIonshell}
Definition~\ref{def:BGI} is the characteristic (operator) formulation of BGI
approximate symmetry, and it is the formulation used throughout this paper.
The BGI method can instead be expressed through the approximate
infinitesimal invariance condition $\widehat X_\varphi(F)\approx_p0$ imposed on
the approximate solution manifold of \eqref{eq:perturbedfamily}. For equivalence with the operator formulation, assume that $F_0=0$ is
regular in the sense of Definition~\ref{def:regular} and that the FS
projection is surjective in the sense of Definition~\ref{def:covering}
below. Expansion of the on-solutions condition then gives the FS symmetry
conditions for the coefficient characteristic, and
Theorem~\ref{thm:correspondence}(ii) reconstructs the operators
in \eqref{eq:BGItier}. The covering assumption ensures that the condition
is tested over every unperturbed jet; it is relevant when differential
identities can obstruct a formal perturbation. The first-order case is
given in \cite{Druzhkov2024}.
\end{remark}

Equating coefficients of $\eps^j$ in \eqref{eq:BGIoperator} gives
\begin{equation}\label{eq:BGItier}
\sum_{a+b=j}\widehat X_{\varphi_a}(F_b)
=
\sum_{a+b=j}\Delta_a(F_b),
\qquad j=0,\ldots,p.
\end{equation}
In particular, $\varphi_0$ is the characteristic of an exact symmetry of $F_0=0$.

\medskip The FS construction instead expands the solution,
\begin{equation}\label{eq:FSexpGeneral}
u_\eps=U_0+\eps U_1+\cdots+\eps^pU_p,
\end{equation}
and substitutes \eqref{eq:FSexpGeneral} into \eqref{eq:perturbedfamily}. The exact equations obtained by setting each coefficient of $\eps^j$ equal to zero constitute the order-$p$ FS system. The coefficient maps in the next section describe its triangular structure.

\section{The higher-order FS system and the BGI--FS correspondence}\label{sec:higherorder}

\subsection{Coefficient maps and triangularity}\label{sec:coefficientmap}

For an $\eps$-dependent differential function $Q[u;\eps]$, define the coefficient map
\begin{equation}\label{eq:Cj}
\mathscr C_j(Q)
:=\frac{1}{j!}\left.\frac{d^j}{d\eps^j}Q[u_\eps;\eps]\right|_{\eps=0},
\qquad j=0,\ldots,p,
\end{equation}
where $u_\eps$ is given by \eqref{eq:FSexpGeneral}. The derivative in \eqref{eq:Cj} is taken after all jet variables $u^a_\alpha$ have been replaced by
\[
\DD_\alpha U_0^a+\eps\DD_\alpha U_1^a+\cdots+\eps^p\DD_\alpha U_p^a.
\]
The maps $\mathscr C_j$ are linear, commute with total derivatives, and satisfy the convolution rule
\begin{equation}\label{eq:convolution}
\mathscr C_j(QR)=\sum_{a+b=j}\mathscr C_a(Q)\mathscr C_b(R).
\end{equation}

The order-$p$ FS system is
\begin{equation}\label{eq:FSsystemGeneral}
S_j:=\mathscr C_j(F)=0,
\qquad j=0,\ldots,p.
\end{equation}

\begin{proposition}[Triangular structure]\label{prop:triangular}
For each $j\geq1$,
\begin{equation}\label{eq:triangular}
S_j=\ellop_{F_0}[U_0](U_j)+R_j(U_0,\ldots,U_{j-1}),
\end{equation}
where $\ellop_{F_0}[U_0]$ is the linearization of $F_0$ at $U_0$. Consequently, every equation $S_j=0$ with $j\geq1$ is linear in its highest coefficient $U_j$ and has the same linear operator in that coefficient.
\end{proposition}

\begin{proof}
The variable $U_j$ enters $u_\eps$ multiplied by $\eps^j$. In the coefficient of $\eps^j$, it can therefore occur only once and only through the first Taylor variation of $F_0$; every product of $U_j$ with a positive-order contribution has order strictly larger than $j$. Terms coming from $F_k$, $k\geq1$, involve only $U_0,\ldots,U_{j-k}$. This gives \eqref{eq:triangular}.
\end{proof}

The proposition states that the same linearization operator is used at every order, whereas the source terms differ.

\begin{definition}[FS-covering]\label{def:covering}
The order-$p$ FS system defines an FS-covering if the projection
\[
(U_0,U_1,\ldots,U_p)\longmapsto U_0
\]
from its infinite prolongation onto the infinite prolongation of $F_0=0$ is surjective.
\end{definition}

For equations in an extended Kovalevskaya form, the triangular equations may be solved successively for the leading derivatives of $U_1,\ldots,U_p$ at the formal local jet level, and the FS-covering property follows. Here ``at the formal local jet level'' means local algebraic solvability in the infinite jet space: once a leading derivative has been selected for $F_0=0$ and all derivatives of the leading derivative have been expressed through the remaining (parametric) jet coordinates, the prolonged equation $S_j=0$, which by Proposition~\ref{prop:triangular} is linear in $U_j$ with the same principal part, is solved in the same way for the corresponding leading derivative of $U_j$ and its derivatives. Every point of the prolonged unperturbed equation, specified by its parametric coordinates, therefore lifts to a point of the prolonged FS system by choosing the parametric coordinates of $U_1,\ldots,U_p$ arbitrarily. No convergence of a formal jet and no global solvability for arbitrary Cauchy data is required or asserted; see \cite[Sect.~2.6]{Olver1993} for the parametric description of prolonged equations in Kovalevskaya form. For gauge or otherwise underdetermined systems, differential identities may impose compatibility conditions on the perturbations, as already occurs at first order \cite{Druzhkov2024}.

\subsection{The expansion of an approximate characteristic}

Let \eqref{eq:approxchar} be an approximate characteristic and define
\begin{equation}\label{eq:theta}
\Theta_j:=\mathscr C_j(\varphi),
\qquad j=0,\ldots,p.
\end{equation}
Then $\Theta=(\Theta_0,\ldots,\Theta_p)$ is a characteristic on the FS variables. For example,
\begin{align}
\Theta_0&=\varphi_0[U_0],\nonumber\\[1ex]
\Theta_1&=\delta_1\varphi_0[U_0]+\varphi_1[U_0],\label{eq:theta12}\\[1ex]
\Theta_2&=\left(\delta_2+\frac12\delta_1^2\right)\varphi_0[U_0]
+\delta_1\varphi_1[U_0]+\varphi_2[U_0],\nonumber
\end{align}
where
\[
\delta_j=\sum_{a,\alpha}\DD_\alpha U_j^a\pd{}{U^a_{0,\alpha}}.
\]

\begin{lemma}[Chain rule for coefficient maps]\label{lem:chain}
Let $\mathbf X_\Theta$ be the evolutionary field on $(U_0,\ldots,U_p)$ with characteristic \eqref{eq:theta}. Then for every differential function or truncated $\eps$-family $Q$,
\begin{equation}\label{eq:chain}
\mathbf X_\Theta\bigl(\mathscr C_j(Q)\bigr)
=
\mathscr C_j\bigl(\widehat X_\varphi(Q)\bigr),
\qquad j=0,\ldots,p.
\end{equation}
\end{lemma}

\begin{proof}
Apply $\mathbf X_\Theta$ to $Q[u_\eps;\eps]$. By the chain rule and the definition of $\Theta_j$,
\[
\mathbf X_\Theta(Q[u_\eps;\eps])
=
\sum_{a,\alpha}\pd{Q}{u^a_\alpha}[u_\eps;\eps]\,
\DD_\alpha\left(\sum_{j=0}^{p}\eps^j\Theta_j^a\right)
\approx_p
\widehat X_\varphi(Q)[u_\eps;\eps].
\]
Taking the coefficient of $\eps^j$ gives \eqref{eq:chain}.
\end{proof}

\subsection{Correspondence theorem}

When an $\eps$-dependent operator in total derivatives $\Delta$ is evaluated at $u_\eps$, its coefficient functions also have expansions. Denote by $\nabla_j$ the coefficient operator obtained at order $j$. The convolution rule then gives
\begin{equation}\label{eq:operatorconv}
\mathscr C_j(\Delta F)=\sum_{a+b=j}\nabla_a(S_b).
\end{equation}

\begin{theorem}[All-orders BGI--FS correspondence]\label{thm:correspondence}
Let $\varphi$ be an approximate characteristic \eqref{eq:approxchar} of order
$p$ and let $\Theta$ be its coefficient characteristic \eqref{eq:theta}.
\begin{enumerate}[label=\textnormal{(\roman*)}]
\item If $\varphi$ is a BGI approximate symmetry characteristic of order $p$
of \eqref{eq:perturbedfamily}, then $\Theta$ is an exact local symmetry
characteristic of the order-$p$ FS system \eqref{eq:FSsystemGeneral}.

\item Conversely, assume that the unperturbed equation $F_0=0$ is regular
(Definition~\ref{def:regular}) and that the order-$p$ FS system defines an
FS-covering (Definition~\ref{def:covering}). If $\Theta$ is an exact local
symmetry characteristic of the FS system, then $\varphi$ is a BGI approximate
symmetry characteristic of order $p$.
\end{enumerate}
\end{theorem}

\begin{proof}
(i) Suppose that \eqref{eq:BGIoperator} holds. Apply $\mathscr C_j$ and use Lemma~\ref{lem:chain} and \eqref{eq:operatorconv}:
\[
\mathbf X_\Theta(S_j)=\sum_{a+b=j}\nabla_a(S_b).
\]
The right-hand side vanishes on the infinite prolongation of the FS system, so $\mathbf X_\Theta$ is an FS symmetry.

\medskip\noindent
(ii) Suppose that $\mathbf X_\Theta$ is an FS symmetry. Put
\[
R_j:=\sum_{a+b=j}\widehat X_{\varphi_a}(F_b),
\qquad j=0,\ldots,p,
\]
so that, before the substitution $u=u_\eps$,
\[
\widehat X_\varphi(F)=\sum_{j=0}^{p}\eps^jR_j+o(\eps^p).
\]
The operators $\Delta_0,\ldots,\Delta_p$ are constructed recursively. For $j=0$, Lemma~\ref{lem:chain} gives
\[
R_0[U_0]=\mathbf X_\Theta(S_0)=0
\]
on the FS system. The covering property implies that $R_0$ vanishes on the unperturbed equation,
and local regularity on a sufficiently high finite jet space yields a
finite-order total-differential operator $\Delta_0$ such that
$R_0=\Delta_0(F_0)$.

Assume that $\Delta_0,\ldots,\Delta_{j-1}$ have been constructed and that
\begin{equation}\label{eq:inductiveresidual}
R_k=\sum_{a+b=k}\Delta_a(F_b),
\qquad k=0,\ldots,j-1.
\end{equation}
Let $\Delta^{<j}=\sum_{a=0}^{j-1}\eps^a\Delta_a$ and consider
\[
Q:=\widehat X_\varphi(F)-\Delta^{<j}(F).
\]
By \eqref{eq:inductiveresidual}, the coefficients of $Q$ below order $j$ vanish identically, while its order-$j$ coefficient is
\[
H_j:=R_j-\sum_{\substack{a+b=j\\a<j}}\Delta_a(F_b).
\]
On the other hand, the coefficient $\mathscr C_j(\widehat X_\varphi F)$ is $\mathbf X_\Theta(S_j)$ by Lemma~\ref{lem:chain}, and the coefficient $\mathscr C_j(\Delta^{<j}F)$ is a linear combination, with total-derivative operators, of $S_0,\ldots,S_j$. Both vanish on the FS system. Hence $\mathscr C_j(Q)=0$ there. Since all coefficients of $Q$ below order $j$ are zero, its order-$j$ coefficient after the substitution is $H_j[U_0]$. Therefore
\[
H_j[U_0]=0
\quad\hbox{on the FS system}.
\]
The function $H_j$ depends only on the original jet variables. Surjectivity of
the FS projection implies that $H_j$ vanishes on $F_0=0$, and the same local
regularity statement gives a finite-order operator $\Delta_j$ such that
$H_j=\Delta_j(F_0)$. This is exactly the order-$j$ identity in \eqref{eq:BGItier}. Induction over $j=0,\ldots,p$ proves \eqref{eq:BGIoperator}.
\end{proof}

\begin{corollary}\label{cor:inclusion}
At every perturbation order, the coefficient map \eqref{eq:theta} embeds BGI approximate local symmetry characteristics into the set of exact local symmetry characteristics of the FS system (part (i) of Theorem~\ref{thm:correspondence}). Under the hypotheses of part (ii), an FS symmetry characteristic corresponds to a BGI symmetry if and only if it is the coefficient expansion of a single approximate characteristic on the original jet space.
\end{corollary}

\begin{remark}[Characteristics as differential functions]\label{rem:literal}
In Theorem~\ref{thm:correspondence} and
Corollary~\ref{cor:inclusion}, characteristics are regarded as differential
functions; no quotient is taken by characteristics vanishing on the
corresponding equation manifolds. Thus equality is literal, and the
coefficient map is injective because $\varphi$ can be recovered from
$\Theta$ by \eqref{eq:theta12} and its higher-order analogues.
\end{remark}

\begin{remark}
The inclusion is known to be strict at first order \cite{Druzhkov2024}. At second order, Corollary~\ref{cor:strict} gives an example in which the second-order FS system of the perturbed cubic wave equation admits an exact dilation whose characteristic is not a coefficient expansion.
\end{remark}

\section{The nonlinear wave equation: first- and second-order scale invariance}\label{sec:wave}

As a particular example, consider a scalar PDE
\begin{equation}\label{eq:wavefull}
\Boxop U+\lambda U^3+\eps F(U)+\eps^2G(U)=0,
\end{equation}
where $\Boxop=\partial_t^2-\partial_x^2-\partial_y^2-\partial_z^2$, $\lambda\neq0$ is a constant coefficient, and $F(U)$, $G(U)$ are constitutive functions. The parameter $\eps$ is held fixed in every symmetry calculation; its reparametrization is used only in the formal equivalence of equation families described below. All constitutive classifications below are local on a connected interval in the $U$-axis. Thus arbitrary real powers are interpreted, for example, on $U>0$, while the logarithmic representatives are written with $\ln|U|$ on an interval not containing $U=0$. The unperturbed equation
\begin{equation}\label{eq:cubicwave}
\Boxop U+\lambda U^3=0
\end{equation}
has the $15$-dimensional conformal point-symmetry algebra in $(1+3)$ dimensions: four translations, six Lorentz transformations, one dilation, and four special conformal transformations. The Poincar\'e subalgebra remains exact for arbitrary $F$ and $G$.

The hypotheses of Theorem~\ref{thm:correspondence}(ii) hold for this family.
Equation \eqref{eq:cubicwave} is in Kovalevskaya form with respect to $t$: it
can be solved for $U_{tt}$, so every differential function vanishing on its
prolongation is obtained by substituting $U_{tt}$ and its total derivatives,
which exhibits it as a total-derivative operator applied to
$\Boxop U+\lambda U^3$; hence \eqref{eq:cubicwave} is regular in the sense of
Definition~\ref{def:regular}. The FS coefficient equations, displayed in
\eqref{eq:FSfirst} and \eqref{eq:FSsecond} below and in general of the form
\eqref{eq:triangular}, are each solvable for $U_{j,tt}$ with a right-hand
side depending only on $U_0,\ldots,U_{j-1}$, their derivatives, and
$U_j$ and its spatial derivatives; they are therefore in extended Kovalevskaya form and define an
FS-covering by the argument following Definition~\ref{def:covering}. The
correspondence theorem thus applies at every order to the wave-equation family
considered in this section. The classification below concerns the dilation symmetry generated by
\begin{equation}\label{eq:D0}
D_0=-x^\mu\partial_{x^\mu}+U_0\partial_{U_0}.
\end{equation}
The opposite overall sign gives the convention used in the original 1989 paper of Fushchych and Shtelen. In \eqref{eq:D0} and below, the summation convention is used, $\mu=0,\ldots,3$,
$x^0=t$, so that
$x^\mu\partial_{x^\mu}=t\,\partial_t+x\,\partial_x+y\,\partial_y+z\,\partial_z$
is the scaling operator in space--time coordinates.

\begin{proposition}[Exact dilation]\label{prop:exactdilation}
For the full equation \eqref{eq:wavefull} to exactly admit the dilation
\[
D=-x^\mu\partial_{x^\mu}+U\partial_U
\]
for all fixed values of $\eps$, it is necessary and sufficient that
\begin{equation}\label{eq:exactcubicFG}
F(U)=a_1U^3,\qquad G(U)=a_2U^3.
\end{equation}
Thus, modulo the perturbation-dependent redefinition of the cubic coefficient in \eqref{eq:eqv2} below, the perturbed family has no nontrivial exact dilation. Under the same condition the full conformal algebra is retained.
\end{proposition}

\begin{proof}
Write the equation as $\Boxop U+H(U;\eps)=0$. The prolongation condition for $D$ is $UH_U-3H=0$. Since $H=\lambda U^3+\eps F+\eps^2G$, splitting with respect to $\eps$ gives $UF'-3F=0$ and $UG'-3G=0$, whose solutions are \eqref{eq:exactcubicFG}. Conversely, a cubic total nonlinearity is the conformally invariant critical wave equation with an $\eps$-dependent cubic coefficient.
\end{proof}

\subsection{First order: the Fushchych--Shtelen theorem}

With $U=U_0+\eps U_1+o(\eps)$, the first-order FS system is given by
\begin{subequations}\label{eq:FSfirst}
\begin{align}
\Boxop U_0+\lambda U_0^3&=0,\label{eq:FSfirsta}\\[1ex]
\Boxop U_1+3\lambda U_0^2U_1+F(U_0)&=0.\label{eq:FSfirstb}
\end{align}
\end{subequations}

\begin{proposition}[Fushchych \& Shtelen, 1989]\label{prop:FSfirst}
A point symmetry of \eqref{eq:FSfirst} whose projection to $(x,U_0)$ is \eqref{eq:D0} has
\begin{equation}\label{eq:firsteta}
\eta_1=aU_0+qU_1+b
\end{equation}
with constants $a,q,b$, and $F$ satisfies
\begin{equation}\label{eq:firstODE}
U_0F'(U_0)-(q+2)F(U_0)+2\lambda aU_0^3+3\lambda bU_0^2=0.
\end{equation}
Consequently,
\begin{equation}\label{eq:firstgeneric}
F(U)=A U^{q+2}+\frac{2\lambda a}{q-1}U^3+\frac{3\lambda b}{q}U^2,
\qquad q\neq0,1,
\end{equation}
while the resonant cases are given by
\begin{align}
q=0:\quad &F(U)=A U^2-2\lambda aU^3-3\lambda bU^2\ln|U|,\label{eq:firstq0}\\[1ex]
q=1:\quad &F(U)=A U^3+3\lambda bU^2-2\lambda aU^3\ln|U|.\label{eq:firstq1}
\end{align}
\end{proposition}

\begin{proof}
The affine form \eqref{eq:firsteta} and the ODE \eqref{eq:firstODE} follow from the more general second-order calculation in Lemma~\ref{lem:secondDE} below after omitting $U_2$ and $G$. Equations \eqref{eq:firstgeneric}--\eqref{eq:firstq1} are the elementary solutions of the Euler equation \eqref{eq:firstODE}.
\end{proof}

To compare with Theorem~1 of Fushchych and Shtelen \cite{Fushchich1989}, multiply the generator by $-1$, identify their leading field $w$ with $U_0$ and their correction $v$ with $U_1$, and set
\[
k=-q,\qquad b_{\rm FS}=-a,\qquad c_{\rm FS}=-b.
\]
Then \eqref{eq:firstgeneric} becomes
\[
F(U)=\frac{2\lambda b_{\rm FS}}{k+1}U^3
+\frac{3\lambda c_{\rm FS}}{k}U^2
+A U^{2-k},
\]
with exactly the two logarithmic cases $k=0,-1$ listed in the 1989 paper. These formulas reproduce the original first-order classification in the present notation.

If affine mixing of the correction variable is suppressed, $a=b=0$. Writing $N=q+2$ gives the normal form
\begin{equation}\label{eq:firstclean}
F(U)=A U^N,
\qquad
D^{(1)}=D_0+(N-2)U_1\partial_{U_1}.
\end{equation}

\subsection{The second-order FS system}

Now set
\begin{equation}\label{eq:FSsecondexp}
U=U_0+\eps U_1+\eps^2U_2+o(\eps^2).
\end{equation}
The second-order FS system is
\begin{subequations}\label{eq:FSsecond}
\begin{align}
E_0&:=\Boxop U_0+\lambda U_0^3=0,\label{eq:E0}\\[1ex]
E_1&:=\Boxop U_1+3\lambda U_0^2U_1+F(U_0)=0,\label{eq:E1}\\[1ex]
E_2&:=\Boxop U_2+3\lambda U_0^2U_2+3\lambda U_0U_1^2
+F'(U_0)U_1+G(U_0)=0.\label{eq:E2}
\end{align}
\end{subequations}

The point symmetries classified below are those of \eqref{eq:FSsecond} that are \emph{projectable} onto the $(x,U_0)$-space, with projection $D_0$: that is, those whose $x$- and $U_0$-components are independent of $U_1$ and $U_2$ and coincide with the corresponding components of \eqref{eq:D0}. Projectability is the condition under which the second-order problem is a continuation of the first-order one.

\begin{lemma}[Determining equations]\label{lem:secondDE}
Let $D^{(2)}$ be a point symmetry of \eqref{eq:FSsecond}, projectable onto the $(x,U_0)$-space with projection $D_0$ \eqref{eq:D0}. Then $D^{(2)}$ has the form
\begin{equation}\label{eq:secondgen}
\begin{aligned}
D^{(2)}={}&-x^\mu\partial_{x^\mu}+U_0\partial_{U_0}
+(aU_0+qU_1+b)\partial_{U_1}\\[1ex]
&+\bigl(rU_0+hU_1+(2q-1)U_2+d\bigr)\partial_{U_2},
\end{aligned}
\end{equation}
where $a,q,b,r,h,d$ are constants. The constitutive functions satisfy
\begin{align}
U F'-(q+2)F+2\lambda aU^3+3\lambda bU^2&=0,\label{eq:FodeUser}\\[1ex]
U G'-(2q+1)G-hF+2\lambda rU^3+3\lambda dU^2+(aU+b)F'&=0.
\label{eq:GodeUser}
\end{align}
The coefficient of the term linear in $U_1$ in the invariance condition for $E_2$ is the derivative of \eqref{eq:FodeUser} and gives no independent condition.
\end{lemma}

\begin{proof}
The calculation is performed with the opposite field $-D^{(2)}$, whose
independent-variable part is $x^\mu\partial_{x^\mu}$ and whose $U_0$-component
is $-U_0$. Write $Y=(Y^0,Y^1,Y^2)=(U_0,U_1,U_2)$, so that this field is
\[
X=x^\mu\partial_{x^\mu}+\eta^A(x,Y)\partial_{Y^A},
\qquad
\eta^0=-U_0 .
\]
Projectability prescribes $\eta^0$ and the $x$-part; the
components $\eta^1,\eta^2$ are functions of $x$ and $Y$ whose form is derived below.

\emph{Step 1: the semilinear identity.} For a semilinear system
$\Boxop Y^A+H^A(Y)=0$, second prolongation of $X$ and restriction to solutions
give
\begin{equation}\label{eq:semilinearidentity}
\Boxop_x\eta^A
+2\eta^A_{,x^\mu B}\,Y^{B\mu}
+\eta^A_{,BC}\,Y^B_\mu Y^{C\mu}
+2H^A-\eta^A_{,B}H^B+\eta^BH^A_{,B}=0,
\end{equation}
where $\Boxop_x$ differentiates the explicit $x$-dependence only. The term
$2H^A$ arises because $\xi^\mu=x^\mu$ contributes $-2Y^A_{\mu\nu}$ to the second
prolongation coefficient.

\emph{Step 2: $\eta^A$ is affine in $Y$ with constant coefficients.} The system
\eqref{eq:FSsecond} constrains second derivatives only, so on its solution set
the first-order jet coordinates $Y^B_\mu$ remain free, and the functions $1$,
$Y^{B\mu}$ and $Y^B_\mu Y^{C\mu}$ are linearly independent there. Splitting
\eqref{eq:semilinearidentity} with respect to them gives $\eta^A_{,BC}=0$ and
$\eta^A_{,x^\mu B}=0$, whence
\[
\eta^A=k^A_B\,Y^B+g^A(x),
\qquad
k^A_B=\const ,
\]
and \eqref{eq:semilinearidentity} reduces to
\begin{equation}\label{eq:reducedidentity}
\Boxop_xg^A+2H^A-k^A_BH^B+\bigl(k^B_CY^C+g^B\bigr)H^A_{,B}=0 .
\end{equation}
Differentiating \eqref{eq:reducedidentity} with respect to $x^\mu$, the only
$x$-dependence being through $g$, yields
\[
\partial_\mu\Boxop_xg^A+(\partial_\mu g^B)\,H^A_{,B}(Y)=0
\qquad\text{for all }Y .
\]
For the system \eqref{eq:FSsecond} one has $g^0=0$, since $\eta^0=-U_0$ is
prescribed. Taking $A=1$ and using $H^1_{,1}=3\lambda U_0^2$, $H^1_{,2}=0$
leaves $\partial_\mu\Boxop_xg^1+3\lambda U_0^2\,\partial_\mu g^1=0$; as $1$ and
$U_0^2$ are linearly independent and $\lambda\neq0$, this gives
$\partial_\mu g^1=0$. Taking $A=2$ and using $H^2_{,2}=3\lambda U_0^2$ then
gives $\partial_\mu g^2=0$ in the same way. Hence every $g^A$ is constant, and
\eqref{eq:reducedidentity} becomes
\begin{equation}\label{eq:constantidentity}
2H^A-\eta^A_{,B}H^B+\eta^BH^A_{,B}=0 .
\end{equation}

\emph{Step 3: splitting.} For the system \eqref{eq:FSsecond},
\begin{align*}
H^0&=\lambda U_0^3,\\[1ex]
H^1&=3\lambda U_0^2U_1+F(U_0),\\[1ex]
H^2&=3\lambda U_0^2U_2+3\lambda U_0U_1^2+F'(U_0)U_1+G(U_0),
\end{align*}
and the affine components can be written as
\[
\eta^1=a'U_0+q'U_1+\rho U_2+b',
\qquad
\eta^2=r'U_0+h'U_1+\rho_2U_2+d' .
\]
Because the system \eqref{eq:FSsecond} constrains only second derivatives, the
variables $U_0,U_1,U_2$ are free coordinates on its solution manifold, so
condition \eqref{eq:constantidentity} must hold identically in
$(U_0,U_1,U_2)$ and may be split with respect to the monomials in $U_1$ and
$U_2$.

For $A=0$, using $\eta^0=-U_0$, $\eta^0_{,0}=-1$, and
$H^0_{,0}=3\lambda U_0^2$, the condition is
\[
2H^0-\eta^0_{,0}H^0+\eta^0H^0_{,0}
=2\lambda U_0^3+\lambda U_0^3-3\lambda U_0^3=0,
\]
so no additional condition arises.

For $A=1$, the derivatives of $H^1$ are $H^1_{,0}=6\lambda U_0U_1+F'(U_0)$,
$H^1_{,1}=3\lambda U_0^2$, and $H^1_{,2}=0$, so the condition reads, written out
in full,
\begin{equation}\label{eq:A1full}
\begin{aligned}
&2H^1-\bigl(a'H^0+q'H^1+\rho H^2\bigr)
-U_0\bigl(6\lambda U_0U_1+F'(U_0)\bigr)\\[1ex]
&\qquad+\bigl(a'U_0+q'U_1+\rho U_2+b'\bigr)\,3\lambda U_0^2=0 .
\end{aligned}
\end{equation}
The variable $U_2$ enters \eqref{eq:A1full} in two terms: the term
$-\rho H^2$ contributes $-3\lambda\rho\,U_0^2U_2$, and the last group
contributes $+3\lambda\rho\,U_0^2U_2$. These two contributions cancel, so
\eqref{eq:A1full} is independent of $U_2$, and the $U_2$-dependence imposes no
condition on $\rho$. The terms of \eqref{eq:A1full} that contain $\rho$ all
come from $-\rho H^2$; they are
\[
-\rho\bigl(3\lambda U_0U_1^2+F'(U_0)\,U_1+G(U_0)\bigr).
\]
Every other term of \eqref{eq:A1full} is at most linear in $U_1$, so the full
coefficient of $U_1^2$ in \eqref{eq:A1full} is $-3\lambda\rho\,U_0$, and the
splitting gives $\rho=0$.

Now collect the coefficient of $U_1^2$ in the $A=2$ condition, taking
$\rho=0$, that is, $\eta^1=a'U_0+q'U_1+b'$. Terms quadratic in $U_1$ arise
from exactly four sources: the term $2H^2$ contributes $6\lambda U_0U_1^2$;
the term $-\rho_2H^2$ contributes $-3\lambda\rho_2\,U_0U_1^2$; the term
$\eta^0H^2_{,0}=-U_0\bigl(6\lambda U_0U_2+3\lambda U_1^2+F''(U_0)U_1+G'(U_0)\bigr)$
contributes $-3\lambda U_0U_1^2$; and the term
$\eta^1H^2_{,1}=\eta^1\bigl(6\lambda U_0U_1+F'(U_0)\bigr)$ contributes
$6\lambda q'\,U_0U_1^2$. Their sum is
\[
3\lambda U_0\,(2q'+1-\rho_2)\,U_1^2 ,
\]
whence $\rho_2=2q'+1$.

Multiplying the generator by $-1$ sends $(a',q',b',r',h',d')$ to
$(-a,-q,-b,-r,-h,-d)$ and the coefficient of $U_2$ in the second component to
$-\rho_2=2q-1$, which gives \eqref{eq:secondgen}. The part of the $A=1$
condition free of $U_1$ is then \eqref{eq:FodeUser}, and the part of the $A=2$
condition free of $U_1$ is \eqref{eq:GodeUser}. The remaining coefficient of
$U_1$ in the $A=2$ condition is
\[
U F''-(q+1)F'+6\lambda aU^2+6\lambda bU,
\]
which is the derivative of \eqref{eq:FodeUser} by $U$, and therefore imposes
no additional constraints.
\end{proof}

\subsection{Solution of the determining system}

Set
\begin{equation}\label{eq:Ndef}
N=q+2.
\end{equation}
For the nonresonant values
\begin{equation}\label{eq:nonresvalues}
N\neq2,\quad N\neq\frac52,\quad N\neq3,
\end{equation}
solving \eqref{eq:FodeUser}--\eqref{eq:GodeUser} gives the following expressions.

\begin{proposition}[Power family]\label{prop:genericfamily}
Under \eqref{eq:nonresvalues},
\begin{equation}\label{eq:FgenericN}
F(U)=A U^N+\frac{2\lambda a}{N-3}U^3+\frac{3\lambda b}{N-2}U^2,
\end{equation}
and
\begin{align}
G(U)={}&B U^{2N-3}
+\frac{A(Na-h)}{N-3}U^N
+\frac{ANb}{N-2}U^{N-1}\nonumber\\[1ex]
&+\frac{\lambda\bigl((N-3)r-ha+3a^2\bigr)}{(N-3)^2}U^3\nonumber\\[1ex]
&+\left[
\frac{3\lambda d}{2N-5}
-\frac{3\lambda hb}{(2N-5)(N-2)}
+\frac{6\lambda ab}{(N-3)(N-2)}
\right]U^2\nonumber\\[1ex]
&+\frac{3\lambda b^2}{(N-2)^2}U.
\label{eq:GgenericN}
\end{align}
\end{proposition}

The admissible family is therefore a sum of six power terms, with coincident exponents at some parameter values. The constants $a,b,r,h,d$ enter terms that can be removed by the equivalence transformations given below.

\subsection{Formal equivalence and normal-form classification}

\begin{definition}[Formal equivalence through second order]\label{def:formaleq}
Two pairs $(F,G)$ are called formally equivalent through second order if the
corresponding one-parameter equation families \eqref{eq:wavefull} are related, modulo
$o(\eps^2)$, by a near-identity transformation of the perturbation parameter,
the cubic coefficient, and the dependent variable of the type listed below.
\end{definition}

Systematic methods for determining equivalence groups of classes of
differential equations are described in
\cite{Bihlo2015equivalence,cheviakov2017symbolic}. In this paper, the
equivalence relation is generated by the three transformation families
below, and the normal forms in Theorem~\ref{thm:classification} are
understood with respect to this relation:
\begin{align}
\widetilde\eps&=\eps+c\eps^2,
&\widetilde F&=F,
&\widetilde G&=G-cF;\label{eq:eqv1}\\[1ex]
\widetilde\lambda&=\lambda+c_1\eps+c_2\eps^2,
&\widetilde F&=F-c_1U^3,
&\widetilde G&=G-c_2U^3;\label{eq:eqv2}\\[1ex]
U&=\widetilde U+c_1\eps+c_2\eps^2,
&\widetilde F&=F+3\lambda c_1U^2,
\nonumber\\[1ex]
&&\widetilde G&=G+c_1F'(U)+3\lambda c_2U^2+3\lambda c_1^2U.
\label{eq:eqv3}
\end{align}
The first two formulas follow by substituting the redefined parameters. For the third, $\Boxop U=\Boxop\widetilde U$
under the constant shift, while the expansion of the nonlinear terms through
order $\eps^2$ is
\begin{align}
&\lambda(\widetilde U+c_1\eps+c_2\eps^2)^3
+\eps F(\widetilde U+c_1\eps+c_2\eps^2)
+\eps^2G(\widetilde U+c_1\eps+c_2\eps^2)\nonumber\\[1ex]
&\quad=\lambda\widetilde U^3
+\eps\bigl[F(\widetilde U)+3\lambda c_1\widetilde U^2\bigr]
+\eps^2\bigl[G(\widetilde U)+c_1F'(\widetilde U)
+3\lambda c_2\widetilde U^2+3\lambda c_1^2\widetilde U\bigr]
+o(\eps^2),\label{eq:eqv3check}
\end{align}
which is \eqref{eq:eqv3}.
In particular, a term proportional to $U^3$ in $F$ is absorbed into the cubic
coefficient, a term proportional to $U^2$ in $F$ is removed by a constant shift
of $U$, and a multiple of $F$ in $G$ is removed by reparametrizing $\eps$; the
linear $U$ term induced by a first-order field shift is removed when that
shift is undone, while the second-order parts remove the remaining $U^3$
and $U^2$ terms in $G$. Logarithms appear when a forcing exponent of an
Euler equation coincides with its homogeneous exponent.

\begin{theorem}[Second-order FS scale-invariance classification]\label{thm:classification}
Assume $F$ is nonzero modulo \eqref{eq:eqv1}--\eqref{eq:eqv3}. Up to these equivalence transformations, the dilation \eqref{eq:D0} is inherited by the second-order FS system, as a point symmetry projectable onto the $(x,U_0)$-space with projection $D_0$, if and only if one of the following cases holds. (Constants $A\neq0$, $B$, and $B_0$ are arbitrary.)

\begin{enumerate}[label=\textnormal{(\roman*)}]
\item For $N\notin\{2,\tfrac52,3\}$,
\beq\label{eq:classgeneric}
\barr
F(U) =A U^N,\qquad G(U)=B U^{2N-3},\\[1ex]
D^{(2)}=-x^\mu\partial_{x^\mu}+U_0\partial_{U_0}
+(N-2)U_1\partial_{U_1}+(2N-5)U_2\partial_{U_2}.
\earr
\eeq

\item For the second-order resonance $N=\tfrac52$,
\beq\label{eq:classfivehalf}
\barr
F(U) =A U^{5/2},\qquad
 G(U) =B U^2\ln|U|,\\[1ex]
D^{(2)}=-x^\mu\partial_{x^\mu}+U_0\partial_{U_0}
+\dfrac12U_1\partial_{U_1}-\dfrac{B}{3\lambda}\partial_{U_2}.
\earr
\eeq

\item For the first-order cubic resonance,
\beq
F(U) =A U^3\ln|U|, \qquad
 G(U) =B U^3\ln^2|U|,\label{eq:classcubiclog}
\eeq
with
\[
\barr
D^{(2)} = -x^\mu\partial_{x^\mu}+U_0\partial_{U_0}
+\left(U_1-\dfrac{A}{2\lambda}U_0\right)\partial_{U_1}\\[1ex]
 \qquad +\left[
U_2+\dfrac{A^2}{4\lambda^2}U_0
-\dfrac{3A^2-4\lambda B}{2\lambda A}U_1
\right]\partial_{U_2}.
\earr
\]

\item For the first-order quadratic resonance,
\beq\label{eq:classquadraticlog}
F(U) =A U^2\ln|U|,\qquad G(U) =\dfrac{A^2}{3\lambda} \left(U\ln^2|U|+U\ln|U|\right)+B_0U,
\eeq
with
\[
D^{(2)}=-x^\mu\partial_{x^\mu}+U_0\partial_{U_0}
-\dfrac{A}{3\lambda}\partial_{U_1}-U_2\partial_{U_2}.
\]
\end{enumerate}
\end{theorem}

\begin{proof}
The first Euler equation \eqref{eq:FodeUser} has a power solution and two resonances $N=2,3$, reproducing Proposition~\ref{prop:FSfirst}. After the removable $U^2$ and $U^3$ terms are eliminated, the nontrivial first-order representatives are $AU^N$ with $N\neq2,3$, $AU^2\ln|U|$, and $AU^3\ln|U|$.

For $F=AU^N$ and $a=b=0$, equation \eqref{eq:GodeUser} is
\[
UG'-(2N-3)G=hAU^N-2\lambda rU^3-3\lambda dU^2.
\]
When $N\neq\tfrac52,3$, the particular solutions are removable term by term:
the $U^N$ term, proportional to $F$, is removed by the parameter
reparametrization \eqref{eq:eqv1}; the $U^3$ term by the $\eps^2$-part of the
cubic-coefficient redefinition \eqref{eq:eqv2}; the $U^2$ term by the
$\eps^2$-part of the field shift \eqref{eq:eqv3}. Since the first-order
representative is already normalized, this second-order shift induces no
$U$ term through order $\eps^2$. This
leaves the homogeneous term $BU^{2N-3}$. A logarithmic particular solution
is required when a nonzero forcing term has the homogeneous exponent. At $N=\tfrac52$, the $U^2$ forcing is resonant and produces $BU^2\ln|U|$, with $d=-B/(3\lambda)$.

For $F=AU^3\ln|U|$, equation \eqref{eq:FodeUser} gives $q=1$, $a=-A/(2\lambda)$, and $b=0$. Substitution in \eqref{eq:GodeUser} shows that the only nonremovable term is $BU^3\ln^2|U|$; matching its logarithmic coefficient determines the displayed $U_1$ coefficient in $\eta_2$, while the constant $U^3$ term determines the displayed $U_0$ coefficient.

For $F=AU^2\ln|U|$, one has $q=0$, $a=0$, and $b=-A/(3\lambda)$. The resonant homogeneous exponent of the $G$ equation is $1$, and direct integration gives
\[
\frac{A^2}{3\lambda}
\left(U\ln^2|U|+U\ln|U|\right)+B_0U.
\]
All remaining terms are removable by \eqref{eq:eqv1}--\eqref{eq:eqv3}. Direct substitution into \eqref{eq:FodeUser}--\eqref{eq:GodeUser} verifies each listed generator.
\end{proof}

The resonances identified in Theorem~\ref{thm:classification} are resonances
of the Euler determining equations \eqref{eq:FodeUser}--\eqref{eq:GodeUser}
for the constitutive functions. When the classified symmetries are used for
reduction in Section~\ref{sec:waves}, a second set of Euler-type equations
arises, namely the invariant-surface conditions for the correction
coefficients; their resonances are distinct in nature but occur at the same
exponents, as shown in Proposition~\ref{prop:solres}.

\begin{corollary}[Strictness in the point-symmetry category]\label{cor:strict}
Let $N\notin\{2,\tfrac52,3\}$ and let $F,G$ be as in \eqref{eq:classgeneric}.
Then $D^{(2)}$ is an exact point symmetry of the FS system
\eqref{eq:FSsecond}, but its point characteristic is not the literal
coefficient expansion \eqref{eq:theta} of any approximate point
characteristic. Consequently the inclusion of Corollary~\ref{cor:inclusion} is
strict at $p=2$ in the point-symmetry category, modulo point characteristics
vanishing on the FS solution manifold; no claim is made in the larger quotient
of all local characteristics.
\end{corollary}

\begin{proof}
The characteristic of $D^{(2)}$ has components
\[
\Theta_0=U_0+x^\mu U_{0,\mu},
\qquad
\Theta_1=(N-2)U_1+x^\mu U_{1,\mu},
\qquad
\Theta_2=(2N-5)U_2+x^\mu U_{2,\mu}.
\]
Suppose $\Theta=\mathscr C(\varphi)$ for an approximate characteristic
$\varphi=\varphi_0+\eps\varphi_1+\eps^2\varphi_2$ on the original jet space. The
first component of \eqref{eq:theta12} gives $\varphi_0[U_0]=\Theta_0$, hence
$\varphi_0(x,u,u_\mu)=u+x^\mu u_\mu$ as a differential function, and therefore
$\delta_1\varphi_0[U_0]=U_1+x^\mu U_{1,\mu}$. The second component of
\eqref{eq:theta12} then requires
\[
\varphi_1[U_0]=\Theta_1-\delta_1\varphi_0[U_0]=(N-3)U_1 .
\]
The left-hand side is a differential function of $U_0$ alone, while for
$N\neq3$ the right-hand side depends on the independent jet variable $U_1$; the
two cannot agree. Since \eqref{eq:FSsecond} constrains second derivatives only, a point
characteristic that vanishes on its solution manifold vanishes identically.
Thus no trivial point characteristic can be added to repair the discrepancy.
\end{proof}

\begin{remark}
Corollary~\ref{cor:strict} answers, at $p=2$, the question left open after
Corollary~\ref{cor:inclusion}. The obstruction follows from
\eqref{eq:theta12}: a coefficient expansion requires the $U_1$-coefficient of
$\Theta_1$ to be exactly $1$, whereas the FS weight of $U_1$ is $N-2$. The two
agree only at $N=3$, which by Proposition~\ref{prop:exactdilation} is the
exactly scale-invariant case, and at which the logarithmic family
\eqref{eq:classcubiclog} must be treated separately.
\end{remark}

\begin{proposition}[Perturbations first appearing at second order]\label{prop:degenerateFS}
Suppose that $F$ is zero modulo \eqref{eq:eqv2}--\eqref{eq:eqv3}, and normalize it to $F=0$. If $G$ is nonzero modulo the same equivalences, the second-order FS system inherits $D_0$ precisely in the following normal forms, where $B\neq0$:
\begin{enumerate}[label=\textnormal{(\alph*)}]
\item for $M\notin\{2,3\}$,
\[
G(U)=BU^M,
\qquad
D^{(2)}=-x^\mu\partial_{x^\mu}+U_0\partial_{U_0}
+\frac{M-1}{2}U_1\partial_{U_1}+(M-2)U_2\partial_{U_2};
\]
\item for $M=2$,
\[
G(U)=BU^2\ln|U|,
\qquad
D^{(2)}=-x^\mu\partial_{x^\mu}+U_0\partial_{U_0}
+\frac12U_1\partial_{U_1}-\frac{B}{3\lambda}\partial_{U_2};
\]
\item for $M=3$,
\[
G(U)=BU^3\ln|U|,
\qquad
D^{(2)}=-x^\mu\partial_{x^\mu}+U_0\partial_{U_0}
+U_1\partial_{U_1}
+\left(U_2-\frac{B}{2\lambda}U_0\right)\partial_{U_2}.
\]
\end{enumerate}
In every case one may add an arbitrary multiple of the symmetry $U_1\partial_{U_2}$.
\end{proposition}

\begin{proof}
With $F=0$, equation \eqref{eq:FodeUser} gives $a=b=0$, while \eqref{eq:GodeUser} becomes
\[
UG'-(2q+1)G+2\lambda rU^3+3\lambda dU^2=0.
\]
Set $M=2q+1$. Away from $M=2,3$, the $U^2$ and $U^3$ particular solutions are removable and the homogeneous solution is $BU^M$. At $M=2$ and $M=3$, the resonant forcings give the two displayed logarithms, with respectively $d=-B/(3\lambda)$ and $r=-B/(2\lambda)$. The coefficient $h$ is unconstrained because $F=0$; it multiplies the symmetry $U_1\partial_{U_2}$ and has been set to zero in the representatives.
\end{proof}

\begin{remark}[Scaling pattern]\label{rem:weights}
In case (i) of Theorem~\ref{thm:classification}, the one-parameter group of
the generator $D^{(2)}$ in \eqref{eq:classgeneric} acts by
$x\mapsto e^{-\tau}x$ and
\[
U_0\mapsto e^{\tau}U_0,
\qquad
U_1\mapsto e^{(N-2)\tau}U_1,
\qquad
U_2\mapsto e^{(2N-5)\tau}U_2 .
\]
The scaling exponents $1$, $N-2$, $2N-5$ of the three fields form an
arithmetic progression with step $N-3$, and the exponents $3$, $N$, $2N-3$ of
the nonlinearities $\lambda U^3$, $AU^N$, $BU^{2N-3}$ in \eqref{eq:classgeneric}
form an arithmetic progression with the same step. Consequently, if the parameter is also formally rescaled,
$\eps\mapsto e^{(3-N)\tau}\eps$, then every term of the expansion
$U_0+\eps U_1+\eps^2U_2$ acquires one and the same factor $e^{\tau}$, so the
expansion as a whole transforms exactly as $U$ does under the unperturbed
dilation \eqref{eq:D0}. The common step $N-3$ vanishes precisely in the
exact-scale-invariant case $N=3$.
\end{remark}

\begin{remark}[Novelty]\label{rem:novelty}
The first-order classification is Theorem~1 of Fushchych and Shtelen \cite{Fushchich1989}. Their final paragraph notes that higher-order FS systems may be considered, but does not give a second-order determining system or classification. Second-order FS classifications for other nonlinear wave families have been published, notably \cite{Zhang2011}. A second-order FS scale-invariance classification for the
specific family \eqref{eq:wavefull} that is complete within the class of
point symmetries projectable onto the $(x,U_0)$-space with prescribed
projection $D_0$, modulo the formal equivalences
\eqref{eq:eqv1}--\eqref{eq:eqv3}, together with the complete classification of BGI point symmetries
continuing the same dilation given below, has not, to the author's knowledge, appeared in the literature.
Symmetries whose $x$-components depend on $U_1$ or $U_2$ are not covered by
Lemma~\ref{lem:secondDE}, and no claim is made about them.
\end{remark}

\subsection{Second-order BGI continuations of the dilation}
\label{sec:BGIwave}

A direct BGI symmetry calculation is now performed for the same PDE family \eqref{eq:wavefull}. The approximate point generator is written as
\begin{equation}\label{eq:BGIexpandedgenerator}
X=X_0+\eps X_1+\eps^2X_2, \qquad
X_0=-x^\mu\partial_{x^\mu}+U\partial_U.
\end{equation}
The correction fields $X_1$ and $X_2$ are allowed to have arbitrary point
components depending on $x$ and $U$. Proposition~\ref{prop:BGIraw} shows
that every admissible pair $(F,G)$ also admits a continuation whose
independent-variable corrections vanish. Such a continuation may be chosen
in the form
\begin{equation}\label{eq:BGIverticalrep}
X_1=(\alpha U+\beta)\partial_U,
\qquad
X_2=(\rho U+\sigma)\partial_U,
\end{equation}
where $\alpha$, $\beta$, $\rho$, and $\sigma$ are constants. This choice is
used for the comparison with the FS point symmetries. The general correction
fields are given in Remark~\ref{rem:BGIvertical}.

\begin{proposition}[Second-order BGI classification]\label{prop:BGIraw}
The dilation $X_0$ admits a second-order BGI point symmetry continuation
\eqref{eq:BGIexpandedgenerator} if and only if the constitutive functions have
the form
\begin{align}
F(U)={}&A_0U^3-2\lambda\alpha U^3L+3\lambda\beta U^2,
\label{eq:BGIFraw}\\[1ex]
G(U)={}&B_0U^3+2\lambda\alpha^2U^3L^2
+2\bigl(\lambda\alpha^2-\alpha A_0-\lambda\rho\bigr)U^3L
\nonumber\\[1ex]
&-6\lambda\alpha\beta U^2L
+\bigl(3A_0\beta-5\lambda\alpha\beta+3\lambda\sigma\bigr)U^2
+3\lambda\beta^2U,
\label{eq:BGIGraw}
\end{align}
where $L\equiv\ln|U|$, and $A_0,B_0$ are arbitrary constants. For every such
pair, \eqref{eq:BGIverticalrep} gives one continuation; all point
continuations are described in Remark~\ref{rem:BGIvertical}.
\end{proposition}

\begin{proof}
Write $X_j=\xi_j^\mu(x,U)\partial_{x^\mu}+\eta_j(x,U)\partial_U$,
$j=1,2$. Here and below $\mu,\nu=0,1,2,3$, repeated space-time indices
are summed, and $j$ is not summed. Splitting the prolonged invariance
condition with respect to the independent derivative coordinates gives
\[
\begin{aligned}
\partial_U\xi_j^\mu&=0,\\
\partial_\mu\xi_{j\nu}+\partial_\nu\xi_{j\mu}&=2\psi_j g_{\mu\nu},\\
\eta_j&=(c_j-\psi_j)U+b_j(x),
\end{aligned}
\]
where $g_{\mu\nu}=\operatorname{diag}(1,-1,-1,-1)$ is the Minkowski
metric, $\xi_{j\nu}=g_{\nu\mu}\xi_j^\mu$,
$\psi_j=(1/4)\partial_\mu\xi_j^\mu$, and $c_j$ is constant.
Indices are raised with $g^{\mu\nu}$.
After $\partial_U\xi_j^\mu=0$ is imposed, the quadratic terms in the
first derivatives give $\partial_U^2\eta_j=0$, while the linear terms give
\[
2\partial^\mu(\partial_U\eta_j)
=\Box\xi_j^\mu=-2\partial^\mu\psi_j.
\]
In four-dimensional Minkowski space, the conformal Killing equations imply
$\psi_j=d_j+2k_{j\mu}x^\mu$, with constant $d_j$ and $k_{j\mu}$;
hence $\Box\psi_j=0$.

Set $H=\lambda U^3+\epsilon F(U)+\epsilon^2G(U)$,
$c=\epsilon c_1+\epsilon^2c_2$,
$b=\epsilon b_1+\epsilon^2b_2$, and
$\psi=-1+\epsilon\psi_1+\epsilon^2\psi_2$.
The remaining invariance condition, through order $\epsilon^2$, is
\[
\Box b+c(UH_U-H)+\psi(3H-UH_U)+bH_U=O(\epsilon^3).
\]
Its coefficient of $\epsilon$ is
\[
UF'-3F+2\lambda c_1U^3+3\lambda b_1U^2+\Box b_1=0.
\]
Differentiation with respect to each $x^\mu$ at fixed $U$ and splitting
with respect to $U$ show, since $\lambda\neq0$, that $b_1$ is constant.
Put $c_1=\alpha$ and $b_1=\beta$. The coefficient of $\epsilon^2$ is then
\[
\begin{aligned}
&UG'-3G+\alpha(UF'-F)+\psi_1(3F-UF')\\
&\qquad+\beta F'+2\lambda c_2U^3+3\lambda b_2U^2+\Box b_2=0.
\end{aligned}
\]
Using the first-order equation, differentiation with respect to each
$x^\mu$ at fixed $U$ gives
\[
\alpha k_{1\mu}=0,\qquad
\partial_\mu(b_2+\beta\psi_1)=0,\qquad
\mu=0,1,2,3.
\]
Consequently $b_2=\sigma-\beta\psi_1$ for a constant $\sigma$, and
$\Box b_2=0$. With $\rho=c_2+\alpha d_1$, the remaining equations become
\begin{align}
UF'-3F+2\lambda\alpha U^3+3\lambda\beta U^2&=0,
\label{eq:BGIFode}\\[1ex]
UG'-3G-\alpha F+2\lambda\rho U^3+3\lambda\sigma U^2
+(\alpha U+\beta)F'&=0.
\label{eq:BGIGode}
\end{align}
The first equation yields \eqref{eq:BGIFraw}.
Substituting it into \eqref{eq:BGIGode} and integrating gives
\eqref{eq:BGIGraw}. Conversely, these functions satisfy the invariance
conditions for the continuation \eqref{eq:BGIverticalrep}.
They therefore give all admissible pairs for arbitrary point corrections
to the prescribed dilation. The determining equations and the
substitutions were checked with the GeM package
\cite{Cheviakov2007,cheviakov2010symbolic}.
\end{proof}

\begin{remark}[General continuations]\label{rem:BGIvertical}
Let
\[
\begin{aligned}
C_j&=\xi_j^\mu\partial_{x^\mu}-\psi_j U\partial_U,\\
\psi_j&=d_j+2k_{j\mu}x^\mu,\qquad j=1,2,
\end{aligned}
\]
be conformal symmetry generators of the unperturbed cubic wave equation,
with $\xi_j$ satisfying the conformal Killing equations in the proof above.
For the functions \eqref{eq:BGIFraw}, \eqref{eq:BGIGraw}, all point
corrections are
\[
\begin{aligned}
X_1&=C_1+(\alpha U+\beta)\partial_U,\\
X_2&=C_2+\bigl[(\rho-\alpha d_1)U+\sigma-\beta\psi_1(x)\bigr]\partial_U,
\end{aligned}
\]
subject to $\alpha k_{1\mu}=0$ for $\mu=0,1,2,3$.
The translation and Lorentz components are unrestricted at both orders.
Constant dilation corrections may be removed by multiplying $X$ by a
scalar series in $\epsilon$ with constant term $1$.
If the first-order perturbation is nontrivial modulo
\eqref{eq:eqv1}--\eqref{eq:eqv3}, then $\alpha\neq0$ and special conformal
terms may occur only in $X_2$. If $\alpha=0$, they may also occur in $X_1$;
when $\beta\neq0$, they require the displayed $x$-dependent additive
correction in $X_2$. The field $C_2$ is unrestricted in either case.
These additional generators do not enlarge the list of admissible
constitutive pairs.
\end{remark}

For each pair in Proposition~\ref{prop:BGIraw}, choose the continuation
\eqref{eq:BGIverticalrep}. Expanding its characteristic according to
Theorem~\ref{thm:correspondence} gives the following point symmetry of the
FS system:
\begin{equation}\label{eq:BGIlift}
\begin{aligned}
D_{\rm lift}^{(2)}={}&-x^\mu\partial_{x^\mu}+U_0\partial_{U_0}
+(U_1+\alpha U_0+\beta)\partial_{U_1}\\[1ex]
&+(U_2+\alpha U_1+\rho U_0+\sigma)\partial_{U_2}.
\end{aligned}
\end{equation}
For this continuation, the constants in Lemma~\ref{lem:secondDE} satisfy
\begin{equation}\label{eq:BGIlockFSconstants}
q=1,
\qquad
a=\alpha,
\qquad b=\beta,
\qquad h=\alpha,
\qquad r=\rho,
\qquad d=\sigma.
\end{equation}
Thus the BGI determining equations \eqref{eq:BGIFode}--\eqref{eq:BGIGode} are the FS equations \eqref{eq:FodeUser}--\eqref{eq:GodeUser} in the case $q=1$, with the additional BGI condition $h=a$.

\begin{remark}[The logarithmic family and the power formula]\label{rem:resonantloss}
The power formulas \eqref{eq:FgenericN}, \eqref{eq:GgenericN} contain factors
$(N-3)^{-1}$ and apply only for $N\neq3$, or $q\neq1$. For the continuation
\eqref{eq:BGIverticalrep}, the FS coefficients satisfy $q=1$ and $h=a$ by
\eqref{eq:BGIlockFSconstants}. Setting $q=1$ and $h=a$ in the FS determining
equations \eqref{eq:FodeUser}, \eqref{eq:GodeUser}, with the constants
identified as in \eqref{eq:BGIlockFSconstants}, gives the BGI determining
equations \eqref{eq:BGIFode}, \eqref{eq:BGIGode}. Integration yields
\eqref{eq:BGIFraw}, \eqref{eq:BGIGraw}. For a nontrivial first-order
perturbation, the equivalence transformations reduce these expressions to
\eqref{eq:BGInormalFG}. This normal form belongs to case (iii) of
Theorem~\ref{thm:classification} with $A=C$ and $B=C^2/(2\lambda)$; in the
full FS family, $B$ is arbitrary.
\end{remark}

The application of the equivalence transformations \eqref{eq:eqv1}--\eqref{eq:eqv3} to Proposition~\ref{prop:BGIraw} removes $A_0U^3$, $3\lambda\beta U^2$, the non-leading $U^3L$ term in $G$, and the remaining $U^3,U^2,U$ terms. This yields the normal form below.

\begin{corollary}[BGI normal form]\label{cor:BGInormal}
Assume that $F$ is nonzero modulo \eqref{eq:eqv1}--\eqref{eq:eqv3}. Up to the equivalence transformations
\eqref{eq:eqv1}--\eqref{eq:eqv3}, the dilation admits a second-order BGI
point symmetry continuation if and only if
\begin{equation}\label{eq:BGInormalFG}
F(U)=C U^3\ln|U|,
\qquad
G(U)=\frac{C^2}{2\lambda}U^3\ln^2|U|,
\qquad C\neq0.
\end{equation}
A normalized BGI generator is
\begin{equation}\label{eq:BGInormalgenerator}
X=X_0-\eps\frac{C}{2\lambda}U\partial_U
+\eps^2\frac{C^2}{4\lambda^2}U\partial_U.
\end{equation}
Its FS lift is
\begin{equation}\label{eq:BGInormallift}
\begin{aligned}
D_{\rm lift}^{(2)}={}&-x^\mu\partial_{x^\mu}+U_0\partial_{U_0}
+\left(U_1-\frac{C}{2\lambda}U_0\right)\partial_{U_1}\\[1ex]
&+\left(U_2-\frac{C}{2\lambda}U_1
+\frac{C^2}{4\lambda^2}U_0\right)\partial_{U_2}.
\end{aligned}
\end{equation}
\end{corollary}

\begin{proof}
Set $C=-2\lambda\alpha$ in \eqref{eq:BGIFraw}. The first- and second-order parts of \eqref{eq:eqv2} and \eqref{eq:eqv3} remove the $U^3$ and $U^2$ terms. Transformation \eqref{eq:eqv1} removes the remaining $U^3\ln|U|$ term in $G$, and the second-order parts of \eqref{eq:eqv2}--\eqref{eq:eqv3} remove the residual $U^3$ and $U^2$ terms. The invariant leading logarithmic coefficient is
\[
2\lambda\alpha^2=\frac{C^2}{2\lambda}.
\]
It remains to determine the second-order continuation constants $\rho,\sigma$
of \eqref{eq:BGIverticalrep} compatible with the normalized pair
\eqref{eq:BGInormalFG}. Substituting $F=CU^3\ln|U|$,
$G=\frac{C^2}{2\lambda}U^3\ln^2|U|$, $\alpha=-\frac{C}{2\lambda}$, and
$\beta=0$ into \eqref{eq:BGIGode}, the logarithmic terms cancel identically
and the remaining terms reduce to
\[
2\lambda\bigl(\rho-\alpha^2\bigr)U^3+3\lambda\sigma U^2=0,
\]
so that $\rho=\alpha^2=\frac{C^2}{4\lambda^2}$ and $\sigma=0$. This gives
\eqref{eq:BGInormalgenerator}, and expansion of its characteristic yields
\eqref{eq:BGInormallift}.
\end{proof}

\begin{corollary}[Perturbations first appearing at second order]\label{cor:BGIdegenerate}
If $F$ is zero modulo equivalence and the first nontrivial perturbation occurs at order $\eps^2$, the nontrivial BGI normal form is
\begin{equation}\label{eq:BGIdegenerateFG}
F(U)=0,\qquad G(U)=C U^3\ln|U|,\qquad C\neq0,
\end{equation}
with a normalized generator
\begin{equation}\label{eq:BGIdegenerategen}
X=X_0-\eps^2\frac{C}{2\lambda}U\partial_U.
\end{equation}
Its FS lift is case (c) of Proposition~\ref{prop:degenerateFS}.
\end{corollary}

\begin{proof}
After $F$ is normalized to zero, \eqref{eq:BGIFode} gives $\alpha=\beta=0$. Equation \eqref{eq:BGIGode} then gives
$G=B_0U^3-2\lambda\rho U^3\ln|U|+3\lambda\sigma U^2$.
The $U^3$ and $U^2$ terms are removable. Setting $C=-2\lambda\rho$ yields \eqref{eq:BGIdegenerateFG}--\eqref{eq:BGIdegenerategen}.
\end{proof}

Corollary~\ref{cor:BGInormal} belongs to case (iii) of Theorem~\ref{thm:classification}. There the FS coefficient $B$ multiplying $U^3\ln^2|U|$ is arbitrary, whereas the BGI version imposes
\begin{equation}\label{eq:BGIlockB}
B=\frac{C^2}{2\lambda}.
\end{equation}
The following table compares the pairs admitting the classified FS point
symmetry with those admitting a second-order BGI point symmetry continuing
the dilation. The BGI calculation allows corrections to both the independent-
and dependent-variable components. All entries are understood modulo
\eqref{eq:eqv1}--\eqref{eq:eqv3}. The first four rows assume a nontrivial
first-order perturbation, so $A$ or $C$ is nonzero; the last three assume a
nontrivial second-order perturbation, so $B\neq0$.
\begin{center}
\begin{tabular}{lll}
\toprule
Representative $F$ & FS-admissible $G$ & BGI-admissible $G$ \\
\midrule
$A U^N$, generic $N$ & $B U^{2N-3}$, $B$ arbitrary & none \\
$A U^{5/2}$ & $B U^2\ln|U|$, $B$ arbitrary & none \\
$C U^3\ln|U|$ & $B U^3\ln^2|U|$, $B$ arbitrary & $B=C^2/(2\lambda)$ \\
$A U^2\ln|U|$ & $\frac{A^2}{3\lambda}(U\ln^2|U|+U\ln|U|)+B_0U$ & none \\
\midrule
$0$ & $B U^M$, $M\notin\{2,3\}$, $B\neq0$ & none \\
$0$ & $B U^2\ln|U|$, $B\neq0$ & none \\
$0$ & $B U^3\ln|U|$, $B\neq0$ & $B\neq0$ \\
\bottomrule
\end{tabular}
\end{center}
The omitted case $F=G=0$ admits both symmetries.

The pairs admitting a BGI point symmetry continuation therefore form a
proper subset of those admitted by the FS dilation classification. For
perturbations nontrivial at first order, the common family is
\eqref{eq:BGInormalFG}: its second-order logarithmic coefficient satisfies
\eqref{eq:BGIlockB}, whereas the corresponding FS coefficient is arbitrary.
When the first nontrivial perturbation occurs at second order,
Corollary~\ref{cor:BGIdegenerate} gives the common family
$F=0$, $G=CU^3\ln|U|$.

\subsection{A power hierarchy at arbitrary order}

The scaling weights determine a family of power perturbations at arbitrary order.

\begin{proposition}\label{prop:powerhierarchy}
Let
\begin{equation}\label{eq:powerhierarchy}
\Boxop U+\lambda U^3+\sum_{j=1}^{p}\eps^jA_jU^{3+j(N-3)}+o(\eps^p)=0.
\end{equation}
Then its order-$p$ FS system admits
\begin{equation}\label{eq:allorderD}
D^{(p)}=-x^\mu\partial_{x^\mu}
+\sum_{j=0}^{p}\bigl(1+j(N-3)\bigr)U_j\partial_{U_j}.
\end{equation}
\end{proposition}

\begin{proof}
Assign $U_j$ the scaling weight $1+j(N-3)$ and each derivative the additional weight $1$. The wave term in the $j$th equation has weight $3+j(N-3)$. Every cubic monomial $U_{j_1}U_{j_2}U_{j_3}$ with $j_1+j_2+j_3=j$ has the same weight. The coefficient of order $j-i$ in the expansion of $U^{3+i(N-3)}$ also has this weight. Hence each equation is homogeneous under \eqref{eq:allorderD}.
\end{proof}

\section{Approximate solutions by reduction with the FS dilation}
\label{sec:waves}

The 1989 letter of Fushchych and Shtelen used the
approximate symmetry to produce approximate solutions of the perturbed equation
in closed form. The same construction is extended here to second order by a joint
Lorentz--dilation reduction of the FS coefficient system.

The proof of Lemma~\ref{lem:secondDE} did not restrict the number of independent
variables. Lemma~\ref{lem:secondDE}, Theorem~\ref{thm:classification} and
Proposition~\ref{prop:degenerateFS} hold in $(1+n)$ dimensions for
$D_0=-x^\mu\partial_{x^\mu}+U_0\partial_{U_0}$, $\mu=0,\ldots,n$. The conformal statements in Proposition~\ref{prop:exactdilation} and the
classification of general BGI point continuations in Section~\ref{sec:BGIwave}
use $(1+3)$-dimensional Minkowski space.

If $X$ is an exact point symmetry of \eqref{eq:wavefull} holding for an arbitrary
$F$ and $G$, then any $X$-invariant reduction of \eqref{eq:wavefull} may be
performed first and expanded in $\eps$ afterwards; the reduced object is a
differential equation with a small parameter, and its ordinary perturbation
expansion coincides with the reduction of the FS system. Therefore no such reduction can
detect the constitutive conditions of Theorem~\ref{thm:classification}.

For \eqref{eq:wavefull} the Poincar\'e algebra is exact for every $F$ and $G$.
Traveling waves $U=U(k\cdot x)$, plane waves, and Lorentz-invariant reductions
$U=U(x_\nu x^\nu)$ are therefore available
whether or not $F$ and $G$ appear in Theorem~\ref{thm:classification}. By Proposition~\ref{prop:exactdilation}, the dilation is an exact symmetry only in the cubic case \eqref{eq:exactcubicFG}. Among the
Poincar\'e generators and the dilation considered here, the dilation therefore
requires an approximate symmetry construction outside that cubic case. (The special conformal generators constitute
a separate problem and are not part of the present reduction.)

\subsection{Reduction by the Lorentz--dilation subalgebra}

Translations $x^\mu\mapsto x^\mu+a^\mu$ are exact point symmetries of
\eqref{eq:wavefull} for every $F$ and $G$; applied simultaneously to
$U_0,U_1,U_2$, they are also point symmetries of the FS system
\eqref{eq:FSsecond}. The construction below may therefore be carried out about
an arbitrary reference point. Fix constants $a^\mu$ and put
\begin{equation}\label{eq:ys}
y^\mu=x^\mu+a^\mu,
\qquad
s=y_\nu y^\nu=(t+a^0)^2-|\mathbf{x}+\mathbf{a}|^2,
\end{equation}
where $y_\mu=\eta_{\mu\nu}y^\nu$ and $\eta=\operatorname{diag}(1,-1,-1,-1)$,
and let $\Omega_\lambda$ be a connected component of
\begin{equation}\label{eq:Omegalambda}
\{y:\lambda s>0\}.
\end{equation}

The FS system is reduced by seven point symmetries. The first six are the
Lorentz rotations and boosts about the point $y=0$,
\begin{equation}\label{eq:Js}
J_{\mu\nu}=y_\mu\partial_{y^\nu}-y_\nu\partial_{y^\mu},
\qquad 0\le\mu<\nu\le3,
\end{equation}
which act on the independent variables only and do not change
$U_0,U_1,U_2$. The seventh is the dilation
\begin{equation}\label{eq:Da2}
D_a^{(2)}=-y^\mu\partial_{y^\mu}+U_0\,\partial_{U_0}
+(aU_0+qU_1+b)\,\partial_{U_1}
+\bigl(rU_0+hU_1+(2q-1)U_2+d\bigr)\,\partial_{U_2},
\end{equation}
with the constants $a,q,b,r,h,d$ of Lemma~\ref{lem:secondDE}. The operator
\eqref{eq:Da2} is the symmetry \eqref{eq:secondgen} with $x$ replaced by
$y=x+a$; since the translations are themselves symmetries, this replacement
again yields a point symmetry of the FS system. The seven operators
\eqref{eq:Js}, \eqref{eq:Da2} span a Lie algebra denoted by
$\mathfrak{g}$; a solution invariant under all seven is called
$\mathfrak{g}$-invariant, and the corresponding truncated expansion
$U_0+\eps U_1+\eps^2U_2$ is called a $\mathfrak{g}$-invariant approximate
solution of \eqref{eq:wavefull}. Here invariance refers to the exact action
of $\mathfrak{g}$ on the FS coefficient system \eqref{eq:FSsecond}. For the constitutive families in
Corollaries~\ref{cor:BGInormal} and~\ref{cor:BGIdegenerate}, the same invariant
expansion also arises from a single approximate point transformation acting on the original
variable $U$; on the remaining cases of Theorem~\ref{thm:classification}
the FS dilation is not the coefficient expansion of any BGI approximate point
transformation (Proposition~\ref{prop:BGIraw}), and the solutions constructed
below are not asserted to be approximately invariant in the BGI sense.
Theorem~\ref{thm:intersection} illustrates the distinction.

Recall the invariance criterion: a solution $U_j=U_j(y)$, $j=0,1,2$, of the FS
system is invariant under a point operator
$X=\xi^\mu(y)\,\partial_{y^\mu}+\sum_j\eta_j\,\partial_{U_j}$
if and only if
\[
\xi^\mu\,\frac{\partial U_j}{\partial y^\mu}=\eta_j,
\qquad j=0,1,2 .
\]
For the operators \eqref{eq:Js} this reads
$y_\mu\,\partial U_j/\partial y^\nu-y_\nu\,\partial U_j/\partial y^\mu=0$ for
all pairs $\mu<\nu$: at every point of $\Omega_\lambda$, the gradient of $U_j$
is parallel to the vector $(y_0,y_1,y_2,y_3)$, which is one half of the
gradient of $s$. Hence $U_j$ and $s$ have parallel gradients, and since the
gradient of $s$ does not vanish on $\Omega_\lambda$, each $U_j$ is locally a
function of $s$ alone. For the dilation \eqref{eq:Da2}, whose coefficients of
$\partial_{y^\mu}$ are $-y^\mu$, the criterion gives the three conditions
\begin{equation}\label{eq:invconds}
y^\mu\frac{\partial U_0}{\partial y^\mu}=-U_0,
\qquad
y^\mu\frac{\partial U_j}{\partial y^\mu}=-\eta_j,
\quad j=1,2,
\end{equation}
with $\eta_1=aU_0+qU_1+b$ and $\eta_2=rU_0+hU_1+(2q-1)U_2+d$. Thus a
$\mathfrak{g}$-invariant solution has no free independent variable left:
the six conditions \eqref{eq:Js} reduce the $U_j$ to functions of the single
variable $s$, and the conditions \eqref{eq:invconds} constrain that dependence
as well. In particular, the first equation of the FS system will reduce to an
algebraic condition rather than a differential one.

For a differentiable function $\phi(y)$, the notation in this section is
\[
\partial_\mu\phi=\frac{\partial\phi}{\partial y^\mu},
\qquad
\partial^\mu\phi=\eta^{\mu\nu}\partial_\nu\phi,
\qquad\hbox{so that}\qquad
\partial_\mu\phi\,\partial^\mu\phi
=\Bigl(\frac{\partial\phi}{\partial t}\Bigr)^{2}-|\nabla\phi|^{2}.
\]
For real solutions, $s>0$ when $\lambda>0$ and $s<0$ when $\lambda<0$.

\begin{lemma}[Reduced system]\label{lem:reduction}
Let $(U_0,U_1,U_2)$ be a real $\mathfrak{g}$-invariant solution of the FS
system \eqref{eq:FSsecond} on $\Omega_\lambda$, with $U_0\not\equiv0$. Then
\begin{equation}\label{eq:Wsol}
U_0=W:=\pm\,(\lambda s)^{-1/2},
\end{equation}
with $s=y_\nu y^\nu$ as in \eqref{eq:ys},
and $U_1=\varphi_1(W)$, $U_2=\varphi_2(W)$, where the functions
$\varphi_1,\varphi_2$ satisfy the first-order equations
\begin{align}
W\dot\varphi_1&=aW+q\varphi_1+b,\label{eq:inv1}\\[1ex]
W\dot\varphi_2&=rW+h\varphi_1+(2q-1)\varphi_2+d,\label{eq:inv2}
\end{align}
together with the second-order equidimensional (Euler) equations
\begin{align}
\Eop[\varphi_1]&=-\frac{F(W)}{\lambda W^2},\label{eq:eul1}\\[1ex]
\Eop[\varphi_2]&=-\frac{3\lambda W\varphi_1^2+F'(W)\varphi_1+G(W)}
                      {\lambda W^2},\label{eq:eul2}
\end{align}
where the dot denotes $d/dW$ and
\begin{equation}\label{eq:Eop}
\Eop:=W^2\frac{d^2}{dW^2}-W\frac{d}{dW}+3,
\qquad
\Eop\bigl[W^m\bigr]=(m^2-2m+3)W^m.
\end{equation}
Moreover, $W$ satisfies identically the first-order relation
\begin{equation}\label{eq:dHam}
\partial_\mu W\,\partial^\mu W=\lambda W^4 .
\end{equation}
\end{lemma}

\begin{proof}
As shown above, invariance under \eqref{eq:Js} gives $U_j=\phi_j(s)$,
$j=0,1,2$, for some functions $\phi_j$ of one variable.

Consider first $U_0=\phi_0(s)$. Since $\partial_\mu s=2y_\mu$, one has
$y^\mu\partial_\mu s=2s$, and the first condition of \eqref{eq:invconds}
becomes $2s\,\phi_0'(s)=-\phi_0(s)$, whose nonzero solutions are
$\phi_0=C_0|s|^{-1/2}$. The constant $C_0$ is determined by substitution into
$\Boxop U_0+\lambda U_0^3=0$. For any function $\phi(s)$,
\[
\Boxop\phi(s)=4s\,\phi''(s)+8\phi'(s),
\]
because $\partial_\mu\phi=2\phi'y_\mu$ and $\Boxop s=8$. Substitution of
$\phi_0=C_0|s|^{-1/2}$ shows that $\Boxop\phi_0=-\lambda\phi_0^3$ holds precisely when
$\lambda C_0^2\,|s|/s=1$, that is, for
$U_0=W=\pm(\lambda s)^{-1/2}$ on $\Omega_\lambda$, which is \eqref{eq:Wsol}.

To verify \eqref{eq:dHam}, differentiation of $W^2=1/(\lambda s)$ gives
$dW/ds=-W/(2s)$, hence
\[
\partial_\mu W=\frac{dW}{ds}\,\partial_\mu s=-\frac{W}{s}\,y_\mu,
\qquad
\partial_\mu W\,\partial^\mu W=\frac{W^2}{s^2}\,y_\mu y^\mu
=\frac{W^2}{s}=\lambda W^4,
\qquad
y^\mu\partial_\mu W=-W,
\]
where the last equality of the middle formula uses $1/s=\lambda W^2$. No separate real square roots of $\lambda$ and $s$ are needed.

On each connected component $\Omega_\lambda$, the map $s\mapsto W$ is smooth
and strictly monotone, hence invertible. A function of $s$ may therefore be
rewritten as a function of $W$, giving
$U_1=\varphi_1(W)$, $U_2=\varphi_2(W)$. This is the form stated in the lemma.
For any function $\phi(W)$, the chain rule and $y^\mu\partial_\mu W=-W$ give
$y^\mu\partial_\mu\phi(W)=-W\dot\phi$, so the two remaining conditions of
\eqref{eq:invconds} become
$-W\dot\varphi_j=-\eta_j$, that is, \eqref{eq:inv1}--\eqref{eq:inv2}.

Finally, for any function $\phi(W)$,
\[
\Boxop\phi(W)=\ddot\phi\;\partial_\mu W\,\partial^\mu W+\dot\phi\;\Boxop W
=\lambda W^4\ddot\phi-\lambda W^3\dot\phi,
\]
using \eqref{eq:dHam} and $\Boxop W=-\lambda W^3$. Substituting
$U_1=\varphi_1(W)$ into \eqref{eq:E1} and $U_2=\varphi_2(W)$ into
\eqref{eq:E2}, and dividing both equations by $\lambda W^2$, gives
\eqref{eq:eul1}--\eqref{eq:eul2}.
\end{proof}

The constitutive normal forms of Theorem~\ref{thm:classification} were defined
locally on a connected interval of the $U$-axis. Accordingly, when a real
noninteger power of $W$ appears below, the sign in \eqref{eq:Wsol} is chosen
so that $W$ lies in that interval; in particular, for representatives
interpreted on $U>0$, the choice is $W=(\lambda s)^{-1/2}>0$.
The choice $W=-(\lambda s)^{-1/2}$ may be retained when the relevant constitutive functions and the powers of
$W$ in the solution formulas are real on $U<0$, as is the case for integer
exponents and for the logarithmic representatives written with $\ln|U|$.

The reduction has three consequences. First, \eqref{eq:dHam} is not an assumption but a \emph{consequence} of the reduction. In \cite{Fushchich1989} the first correction is sought in the form $v=f(w)$,
which turns the equation for the correction into an ordinary differential
equation only if $\partial_\mu w\,\partial^\mu w$ is a function of $w$; the
required relation is imposed there as a compatibility condition, and
$\{\Boxop w+\lambda w^3=0,\ \partial_\mu w\,\partial^\mu w=\lambda w^4\}$
is the d'Alembert--Hamilton system. For the explicit Lorentz-radial family,
Lemma~\ref{lem:reduction} explains why that substitution closes and why the
relation $\partial_\mu w\,\partial^\mu w=\lambda w^4$ appears: it is produced
by the joint Lorentz--dilation reduction. The relation is also consistent with
the scaling: under the dilation group, $w$ is multiplied by $e^{\tau}$ and each
differentiation contributes a further factor $e^{\tau}$, so both sides of
\eqref{eq:dHam} are multiplied by $e^{4\tau}$. This establishes a
group-theoretic derivation of the displayed Fushchych--Shtelen family, but it
does not assert that every solution of their d'Alembert--Hamilton system is
Lorentz--dilation invariant.

Second, the same operator $\Eop$ appears in both \eqref{eq:eul1} and
\eqref{eq:eul2}. This reflects Proposition~\ref{prop:triangular}: every
equation of the FS system after the first contains the same linearization of
the unperturbed equation, here reduced to $\Eop$. For real power forcing, $\Eop$ has no resonance: its indicial polynomial $m^2-2m+3$ has discriminant $-8$, so its
roots are $1\pm i\sqrt2$ and its homogeneous solutions are the oscillatory pair
\begin{equation}\label{eq:oscmodes}
W\cos\bigl(\sqrt2\ln|W|\bigr),
\qquad
W\sin\bigl(\sqrt2\ln|W|\bigr).
\end{equation}
These homogeneous solutions are excluded by
\eqref{eq:inv1}--\eqref{eq:inv2}: they are not $\mathfrak{g}$-invariant, and
the invariant solution is the particular solution obtained by discarding them.
No such terms appear in the formulas of \cite{Fushchich1989}, which is
consistent with, and explained by, their invariance.

Third, the exceptional exponents of the classification
appear in the conditions \eqref{eq:inv1}--\eqref{eq:inv2}.

\begin{proposition}[Possible resonances of the reduction]\label{prop:solres}
The homogeneous solution of \eqref{eq:inv1} is $\varphi_1\propto W^q$, and that
of \eqref{eq:inv2} is $\varphi_2\propto W^{2q-1}$. The degree of an
inhomogeneous term of \eqref{eq:inv1} can coincide with the homogeneous
exponent $q$ only for $q\in\{0,1\}$, and in \eqref{eq:inv2} with $2q-1$ only
for $q\in\{\tfrac12,1\}$. A logarithmic term is generated at such a value
whenever the coefficient of the corresponding inhomogeneous term is nonzero.
In terms of \eqref{eq:Ndef}, the union of these values is
\[
N\in\{2,\tfrac52,3\},
\]
which is the exceptional set \eqref{eq:nonresvalues} of
Theorem~\ref{thm:classification}.
\end{proposition}

\begin{proof}
The inhomogeneous terms of \eqref{eq:inv1} are $aW$ and $b$, of degrees $1$
and $0$ in $W$; coincidence with the degree $q$ of the homogeneous solution
occurs iff $q=1$ or $q=0$. In \eqref{eq:inv2} the inhomogeneous terms are
$rW$, the terms of $h\varphi_1$ --- whose degrees are $q$, $1$, and $0$ ---
and $d$; coincidence with $2q-1$ occurs iff $2q-1\in\{1,q,0\}$, that is, iff
$q=1$ or $q=\tfrac12$. The correspondence $N=q+2$ of \eqref{eq:Ndef} gives the
stated values.
\end{proof}

\subsection{The power family}

\begin{lemma}\label{lem:nomixing}
In case (i) of Theorem~\ref{thm:classification} the generator constants satisfy
$a=b=r=h=d=0$.
\end{lemma}

\begin{proof}
For $F=AU^N$ with $N\notin\{2,3\}$, equation \eqref{eq:FodeUser} reads
$A(N-q-2)U^N+2\lambda aU^3+3\lambda bU^2=0$, forcing $q=N-2$ and $a=b=0$. With
$a=b=0$, \eqref{eq:GodeUser} becomes
$UG'-(2N-3)G=hAU^N-2\lambda rU^3-3\lambda dU^2$, whose solution is $BU^{2N-3}$
plus multiples of $U^N$, $U^3$ and $U^2$ with coefficients proportional to $h$,
$r$ and $d$ respectively; the normal form \eqref{eq:classgeneric} requires all
three to vanish.
\end{proof}

\begin{theorem}[Invariant approximate solution for the power family]\label{thm:generictower}
Let $N\notin\{2,\tfrac52,3\}$ and let $F,G$ be as in \eqref{eq:classgeneric}.
Then the $\mathfrak{g}$-invariant approximate solution of \eqref{eq:wavefull} is
\begin{equation}\label{eq:generictower}
U=W+\eps C_1W^{N-2}+\eps^2C_2W^{2N-5}+o(\eps^2),
\end{equation}
with $W$ given by \eqref{eq:Wsol} and
\begin{equation}\label{eq:genericconsts}
C_1=\frac{-A}{\lambda\,\Qpoly(N)},
\qquad
C_2=\frac{-\bigl(3\lambda C_1^2+ANC_1+B\bigr)}{2\lambda\,\Rpoly(N)},
\end{equation}
where
\begin{equation}\label{eq:QR}
\Qpoly(N)=N^2-6N+11,
\qquad
\Rpoly(N)=2N^2-12N+19 .
\end{equation}
Both $\Qpoly$ and $\Rpoly$ have discriminant $-8$ and are therefore never zero
for real $N$; the construction imposes no restriction beyond
\eqref{eq:nonresvalues}.
\end{theorem}

\begin{proof}
By Lemma~\ref{lem:nomixing} the conditions \eqref{eq:inv1}--\eqref{eq:inv2} are
homogeneous, $W\dot\varphi_1=(N-2)\varphi_1$ and
$W\dot\varphi_2=(2N-5)\varphi_2$, so $\varphi_1=C_1W^{N-2}$ and
$\varphi_2=C_2W^{2N-5}$. By \eqref{eq:Eop}, $\Eop[W^{N-2}]=\Qpoly(N)W^{N-2}$ and
$\Eop[W^{2N-5}]=2\Rpoly(N)W^{2N-5}$. Equation \eqref{eq:eul1} then gives
$\lambda\Qpoly(N)C_1=-A$. Substituting $\varphi_1$ into the right-hand side of
\eqref{eq:eul2} reduces that expression to a single power,
$-\bigl(3\lambda C_1^2+ANC_1+B\bigr)W^{2N-5}/\lambda$, and \eqref{eq:eul2}
gives the stated $C_2$. The discriminants are $36-44=-8$ and $144-152=-8$.
\end{proof}

The exponents in \eqref{eq:generictower} match the scaling pattern of
Remark~\ref{rem:weights}. Under the dilation group $y\mapsto e^{-\tau}y$ one
has $s\mapsto e^{-2\tau}s$, hence $W=\pm(\lambda s)^{-1/2}\mapsto e^{\tau}W$,
and therefore
\[
W^{N-2}\mapsto e^{(N-2)\tau}W^{N-2},
\qquad
W^{2N-5}\mapsto e^{(2N-5)\tau}W^{2N-5}:
\]
the three terms of the invariant solution scale with exactly the factors that
Remark~\ref{rem:weights} assigns to $U_0$, $U_1$, $U_2$.

\subsection{The exceptional exponents}

At the three exceptional exponents, logarithms occur in the invariance
conditions \eqref{eq:inv1}--\eqref{eq:inv2} whenever the corresponding
resonant forcing coefficient is nonzero. Write
\begin{equation}\label{eq:Ldef}
L:=\ln|W| .
\end{equation}

\begin{theorem}[Resonant invariant approximate solutions]\label{thm:resonanttower}
For the resonant cases (ii), (iii), (iv) of Theorem~\ref{thm:classification},
the $\mathfrak{g}$-invariant approximate solution of \eqref{eq:wavefull} is
$U=W+\eps\varphi_1+\eps^2\varphi_2+o(\eps^2)$ with $W$ given by \eqref{eq:Wsol}
and
\begin{enumerate}[label=\textnormal{(\alph*)}]
\item for $N=\tfrac52$, with $F,G$ as in \eqref{eq:classfivehalf},
\begin{equation}\label{eq:towerfivehalf}
\varphi_1=-\frac{4A}{9\lambda}W^{1/2},
\qquad
\varphi_2=-\frac{B}{3\lambda}L+\frac{14A^2-18\lambda B}{81\lambda^2};
\end{equation}

\item for $N=3$, with $F,G$ as in \eqref{eq:classcubiclog},
\begin{equation}\label{eq:towercubiclog}
\varphi_1=-\frac{A}{2\lambda}WL,
\qquad
\varphi_2=p\,W\bigl(L^2-1\bigr)+\frac{A^2}{4\lambda^2}WL,
\qquad
p=\frac{3A^2}{8\lambda^2}-\frac{B}{2\lambda};
\end{equation}

\item for $N=2$, with $F,G$ as in \eqref{eq:classquadraticlog},
\begin{equation}\label{eq:towerquadraticlog}
\varphi_1=-\frac{A}{3\lambda}L-\frac{2A}{9\lambda},
\qquad
\varphi_2=\frac{2A^2-27\lambda B_0}{162\lambda^2}\,W^{-1}.
\end{equation}
\end{enumerate}
\end{theorem}

\begin{proof}
In each case the generator constants are read off
Theorem~\ref{thm:classification}: for (a), $q=\tfrac12$, $a=b=r=h=0$ and
$d=-B/(3\lambda)$; for (b), $q=1$, $a=-A/(2\lambda)$, $b=d=0$,
$r=A^2/(4\lambda^2)$ and $h=(4\lambda B-3A^2)/(2\lambda A)$; for (c), $q=0$,
$b=-A/(3\lambda)$ and $a=r=h=d=0$.

For (a), \eqref{eq:inv1} is $W\dot\varphi_1=\tfrac12\varphi_1$, so
$\varphi_1=C_1W^{1/2}$, and $\Eop[W^{1/2}]=\tfrac94W^{1/2}$ gives
$C_1=-4A/(9\lambda)$ from \eqref{eq:eul1}. Since $2q-1=0$, equation
\eqref{eq:inv2} degenerates to the quadrature $W\dot\varphi_2=d$, whence
$\varphi_2=-\tfrac{B}{3\lambda}L+C_2$: the logarithmic form is already
dictated by the invariant-surface equation; the reduced wave equation, whose
source also contains $\ln|W|$, confirms its coefficient and determines the
additive constant. Using $\Eop[L]=3L-2$ and $\Eop[1]=3$ in
\eqref{eq:eul2}, the terms in $L$ cancel and the constant terms give
$C_2=(14A^2-18\lambda B)/(81\lambda^2)$.

For (b), $q=1$ and \eqref{eq:inv1} reads $W\dot\varphi_1-\varphi_1=aW$, which is
resonant; its solutions are $\varphi_1=aWL+\kappa W$, and \eqref{eq:eul1} with
$\Eop[WL]=2WL$ and $\Eop[W]=2W$ forces $\kappa=0$. Then $2q-1=1$ and
\eqref{eq:inv2} reads $W\dot\varphi_2-\varphi_2=rW+h\varphi_1$, whose forcing
contains both $W$ and $WL$ against the homogeneous solution $W$; writing
$\varphi_2=pWL^2+p_1WL+mW$ gives $p=-hA/(4\lambda)$ and $p_1=A^2/(4\lambda^2)$.
Substituting into \eqref{eq:eul2}, with $\Eop[WL^2]=2WL^2+2W$, reproduces the
same $p$ and $p_1$ and additionally determines $m=-p$. Eliminating $h$ gives the
stated $p$.

For (c), $q=0$ and \eqref{eq:inv1} degenerates to $W\dot\varphi_1=b$, giving
$\varphi_1=-\tfrac{A}{3\lambda}L+\kappa$, and \eqref{eq:eul1} determines
$\kappa=-2A/(9\lambda)$. Since $2q-1=-1$ and $r=h=d=0$, equation
\eqref{eq:inv2} is the \emph{homogeneous} equation $W\dot\varphi_2+\varphi_2=0$,
so $\varphi_2=C_2W^{-1}$, and $\Eop[W^{-1}]=6W^{-1}$ in \eqref{eq:eul2} gives
the stated constant.
\end{proof}

In case (c), $F$ and $G$ in \eqref{eq:classquadraticlog} and the first
correction $\varphi_1$ are logarithmic, whereas the second correction
$\varphi_2$ is a pure power.
The reason is a cancellation in the source of \eqref{eq:eul2}: with
$\varphi_1=-\tfrac{A}{3\lambda}\bigl(L+\tfrac23\bigr)$,
\[
3\lambda W\varphi_1^2+F'(W)\varphi_1+G(W)
=\frac{A^2}{3\lambda}W\Bigl[
\bigl(L+\tfrac23\bigr)^2-(2L+1)\bigl(L+\tfrac23\bigr)+L^2+L
\Bigr]+B_0W,
\]
and the bracket reduces identically to the constant $-\tfrac29$, every term in
$L^2$ and $L$ cancelling. The cancellation of these logarithmic terms leaves a second-order
correction proportional to $W^{-1}$, with its coefficient determined by
\eqref{eq:eul2}.

\subsection{Recovery of the Fushchych--Shtelen solutions}

The first-order rows of Theorems~\ref{thm:generictower} and
\ref{thm:resonanttower} are the normal-form representatives of the solutions on
the last page of \cite{Fushchich1989}. The full formulas are recovered by
retaining the affine mixing constants in Proposition~\ref{prop:FSfirst}.

\begin{proposition}[Full first-order table]\label{prop:FSrecovery}
Use the sign convention of \cite{Fushchich1989} and put
\[
k=-q,
\qquad b_{\rm FS}=-a,
\qquad c_{\rm FS}=-b.
\]
For the Lorentz-radial leading field \eqref{eq:Wsol}, the joint
Lorentz--dilation invariant correction is
\begin{align*}
\varphi_1={}&-\frac{A}{\lambda(k^2+2k+3)}W^{-k}
-\frac{b_{\rm FS}}{k+1}W-\frac{c_{\rm FS}}{k},
&&k\neq0,-1,\\[1ex]
\varphi_1={}&-c_{\rm FS}L-b_{\rm FS}W
-\frac13\left(2c_{\rm FS}+\frac{A}{\lambda}\right),
&&k=0,\\[1ex]
\varphi_1={}&-W\left(\frac{A}{2\lambda}+b_{\rm FS}L\right)+c_{\rm FS},
&&k=-1.
\end{align*}
These are exactly the three correction formulas displayed in
\cite{Fushchich1989}. Setting $b_{\rm FS}=c_{\rm FS}=0$ gives the normalized
generic row of Theorem~\ref{thm:generictower}. In the two resonant rows the
normal forms of Theorem~\ref{thm:classification} retain only the logarithmic
term of $F$, which is carried by $c_{\rm FS}$ at $k=0$ and by $b_{\rm FS}$ at
$k=-1$; the normalized rows of Theorem~\ref{thm:resonanttower} therefore
correspond to $A=b_{\rm FS}=0$ with $3\lambda c_{\rm FS}$ in the role of the
constant of \eqref{eq:classquadraticlog}, and to $A=c_{\rm FS}=0$ with
$2\lambda b_{\rm FS}$ in the role of the constant of \eqref{eq:classcubiclog},
respectively.
\end{proposition}

\begin{proof}
The invariant-surface equation \eqref{eq:inv1} gives the affine terms directly.
For $q\neq0,1$ its general invariant solution is
\[
\varphi_1=CW^q+\frac{a}{1-q}W-\frac{b}{q},
\]
and \eqref{eq:eul1} determines
$C=-A/[\lambda(q^2-2q+3)]$. The equations at $q=0$ and $q=1$ are integrated
before division and give the two logarithmic formulas. Substitution of
$k=-q$, $b_{\rm FS}=-a$, and $c_{\rm FS}=-b$ yields the displayed table.
\end{proof}

The second-order corrections $\varphi_2$ therefore extend the normalized
first-order solutions to second order. The reduction is uniform in the
order: by \eqref{eq:dHam} the identity
$\Boxop\phi(W)=\lambda W^4\ddot\phi-\lambda W^3\dot\phi$ holds for every
function $\phi$, so no new differential constraint arises at later orders and
the same operator $\Eop$ governs every correction.

\subsection{Perturbations first appearing at second order}

The normal forms of Proposition~\ref{prop:degenerateFS} also reduce explicitly.
In these cases the first correction vanishes: the invariant-surface equation
allows a power of $W$, but the homogeneous Euler equation
$\Eop[\varphi_1]=0$ removes it for real scaling weights.

Theorem~\ref{thm:degeneratetower} below does not appear in
\cite{Fushchich1989}, which treats first-order perturbations only. Upon
renaming $\delta=\eps^2$, the perturbed equation becomes first-order in
$\delta$, and the resulting formulas for $\varphi_2$ reduce to the first-order
table of Proposition~\ref{prop:FSrecovery}; the additional second-order content
of the theorem is the proof that the order-$\eps$ correction vanishes,
together with the underlying classification of
Proposition~\ref{prop:degenerateFS}.

\begin{theorem}[Invariant approximate solutions for $F=0$]\label{thm:degeneratetower}
Let $F=0$ and let $G$ be one of the normal forms in
Proposition~\ref{prop:degenerateFS}. Then the corresponding
$\mathfrak g$-invariant approximate solution is
\[
U=W+\eps^2\varphi_2(W)+o(\eps^2),
\]
where
\begin{enumerate}[label=\textnormal{(\alph*)}]
\item for $M\notin\{2,3\}$ and $G(U)=BU^M$,
\[
\varphi_2=-\frac{B}{\lambda\,(M^2-6M+11)}W^{M-2};
\]
\item for $G(U)=BU^2\ln|U|$,
\[
\varphi_2=-\frac{B}{3\lambda}\left(L+\frac23\right);
\]
\item for $G(U)=BU^3\ln|U|$,
\[
\varphi_2=-\frac{B}{2\lambda}WL.
\]
\end{enumerate}
\end{theorem}

\begin{proof}
In all three cases \eqref{eq:eul1} is homogeneous. Compatibility with the
invariant-surface equation forces $\varphi_1=0$. In case (a),
\eqref{eq:inv2} gives $\varphi_2=CW^{M-2}$ and
$\Eop[W^{M-2}]=(M^2-6M+11)W^{M-2}$, which determines $C$. In case (b),
\eqref{eq:inv2} gives
$\varphi_2=-B(3\lambda)^{-1}L+C$; using
$\Eop[L]=3L-2$ in \eqref{eq:eul2} determines
$C=-2B/(9\lambda)$. In case (c), the resonant invariance equation gives
$\varphi_2=-B(2\lambda)^{-1}WL+C W$, and
$\Eop[WL]=2WL$, $\Eop[W]=2W$ force $C=0$.
\end{proof}

\subsection{The intersection with the BGI construction}

For perturbations nontrivial at first order, Corollary~\ref{cor:BGInormal}
gives the pair \eqref{eq:BGInormalFG} admitted by both classifications.
It is case (iii) of Theorem~\ref{thm:classification} with $A=C$ and
$B=C^2/(2\lambda)$. For this value, the generator constants of
Theorem~\ref{thm:resonanttower}(b) become
\[
a=-\frac{C}{2\lambda},
\qquad
r=\frac{C^2}{4\lambda^2},
\qquad
h=\frac{4\lambda B-3C^2}{2\lambda C}=-\frac{C}{2\lambda}=a,
\qquad
b=d=0,
\]
so that $D^{(2)}$ is the FS point symmetry \eqref{eq:BGInormallift}
obtained by expanding the characteristic of the generator
\eqref{eq:BGInormalgenerator}, and the condition $h=a$ in
\eqref{eq:BGIlockFSconstants} holds. The
corresponding invariant solution is therefore approximately invariant in both
frameworks simultaneously.

\begin{theorem}[The intersection solution]\label{thm:intersection}
Let $F$ and $G$ be as in \eqref{eq:BGInormalFG} and put $\mu:=\eps C/(2\lambda)$.
Then the $\mathfrak{g}$-invariant approximate solution of \eqref{eq:wavefull} is
\begin{equation}\label{eq:intersectionsol}
U=W-\mu\,W\ln|W|
+\frac{\mu^2}{2}W\Bigl[\bigl(\ln|W|+1\bigr)^2-2\Bigr]+o(\eps^2),
\end{equation}
and this expression is the second-order Taylor polynomial in $\mu$ of the
power law with a perturbation-dependent exponent
\begin{equation}\label{eq:anomalous}
U=\Bigl(1-\frac{\mu^2}{2}\Bigr)\operatorname{sgn}(W)\,|W|^{\,1-\mu+\mu^2}+O(\mu^3),
\end{equation}
which is real-valued for both signs of $W$ in \eqref{eq:Wsol}.
\end{theorem}

\begin{proof}
Set $A=C$ and $B=C^2/(2\lambda)$ in \eqref{eq:towercubiclog}. Then
$p=3C^2/(8\lambda^2)-C^2/(4\lambda^2)=C^2/(8\lambda^2)$, so that
\[
\varphi_2=\frac{C^2}{8\lambda^2}W\bigl(L^2-1\bigr)+\frac{C^2}{4\lambda^2}WL
=\frac{C^2}{8\lambda^2}W\bigl(L^2+2L-1\bigr),
\]
which gives \eqref{eq:intersectionsol} on substituting $\eps\varphi_1=-\mu WL$,
$\eps^2\varphi_2=\tfrac{\mu^2}{2}W(L^2+2L-1)$ and $L^2+2L-1=(L+1)^2-2$. To obtain \eqref{eq:anomalous}, expand
$\operatorname{sgn}(W)|W|^{1-\mu+\mu^2}=W\exp\bigl[(-\mu+\mu^2)L\bigr]
=W\bigl[1-\mu L+\mu^2L+\tfrac{\mu^2}{2}L^2\bigr]+O(\mu^3)$,
with $L=\ln|W|$ as in \eqref{eq:Ldef},
and multiply by $1-\tfrac{\mu^2}{2}$.
\end{proof}

Equation \eqref{eq:anomalous} expresses the logarithmic corrections as the
second-order expansion of a power whose exponent depends on the perturbation
parameter. For the pair \eqref{eq:BGInormalFG}, the generator
\eqref{eq:BGInormalgenerator} and the associated FS point symmetry
\eqref{eq:BGInormallift} give this invariant solution in both frameworks.
This pair belongs to the cubic-logarithmic FS family, with $A=C$ and the
additional condition $B=C^2/(2\lambda)$. The FS classification also admits the
remaining pairs of Theorem~\ref{thm:classification}.

\begin{remark}[Verification]\label{rem:sec5verify}
Lemma~\ref{lem:reduction} and Theorems~\ref{thm:generictower},
\ref{thm:resonanttower}, \ref{thm:degeneratetower}, and
\ref{thm:intersection} were verified symbolically:
for each listed case, the three residuals of the FS system
\eqref{eq:E0}--\eqref{eq:E2} and both invariant-surface conditions
\eqref{eq:inv1}--\eqref{eq:inv2} were evaluated on \eqref{eq:Wsol} in
$(1+3)$-dimensional Minkowski coordinates and found to vanish identically.
\end{remark}

\subsection{Validity of the Lorentz--radial expansion}\label{sec:radialvalidity}

The FS equations determine the coefficients of a formal expansion
\begin{equation}\label{eq:formalvalidity}
U=U_0+\eps U_1+\eps^2U_2+o(\eps^2).
\end{equation}
Here, \emph{formal} means that the coefficients through order $\eps^2$
vanish after substitution into the original equation. The expansion is taken
with the independent variables held fixed; this construction does not by
itself imply uniform accuracy on domains that grow as $\eps\to0$.

The Lorentz-radial solutions found above contain no explicit polynomial factor
in $t$. Their correction terms are functions of
$W=\pm(\lambda s)^{-1/2}$ and $\ln|W|$, where
$s=t^2-x^2-y^2-z^2$. Their use is nevertheless restricted to regions in which
the perturbation ordering is preserved. In particular, one must require
\begin{equation}\label{eq:orderingconditions}
|\eps\varphi_1(W)|\ll|W|,
\qquad
|\eps^2\varphi_2(W)|\ll|W|.
\end{equation}
When $\varphi_1\not\equiv0$, one may additionally impose the
hierarchy $|\eps^2\varphi_2(W)|\ll|\eps\varphi_1(W)|$ uniformly on a chosen
subdomain on which $|\varphi_1|$ stays bounded away from zero. This comparison
is not imposed on the solutions of
Theorem~\ref{thm:degeneratetower}, where $\varphi_1\equiv0$, and it
fails near an isolated zero of $\varphi_1$ when $\varphi_2$ is nonzero
there, even if the approximation remains useful; near such a zero a
separate local ordering, in which $\eps^2\varphi_2$ is compared with $W$ only,
is the appropriate requirement. Theorem~\ref{thm:resonanttower}(c) gives a
concrete instance: $\varphi_1$ is logarithmic in $W$ and vanishes at
$|W|=e^{-2/3}$, while $\varphi_2\propto W^{-1}$. If the coefficient of
$W^{-1}$ is nonzero, the hierarchy fails both near that zero and as $W\to0$. One must in
addition remain away from the singular hypersurface $s=0$. These conditions
state the validity domain of the displayed second-order Lorentz-radial
solutions. The next section turns to a periodic reduction in which the loss of
uniformity can be seen explicitly and removed by phase renormalization.

\section{Long-scale validity of the FS expansion: cnoidal waves and phase renormalization}\label{sec:cnoidal}

Traveling waves reduce \eqref{eq:wavefull} by the exact translation symmetry
$\partial_t+c\,\partial_x$. The reduction exists for arbitrary $F$ and $G$ and
therefore does not detect the scale-invariance conditions of
Theorem~\ref{thm:classification}. Its role here is different: it provides a
tractable periodic orbit on which the coefficient equations of the FS system
can be followed over many wavelengths. The resulting calculation shows
explicitly how a formally correct FS expansion loses uniformity and how its
secular terms are resummed into an amplitude-dependent wavenumber.

\subsection{Secular growth and the traveling-wave test}\label{sec:secular}

A correction is called \emph{secular} when it grows with time or distance until
it is no longer smaller than the preceding term of the perturbation expansion.
The first-order FS solution given in Eq.~(5.11) of \cite{Tarayrah2023} for
$U_{tt}=(1+\eps U_x^2)U_{xx}$ gives a direct example. Its correction contains a
term proportional to
\begin{equation}\label{eq:secularexample}
\eps t(x-t)^3\exp\!\left[-3(x-t)^2\right].
\end{equation}
For fixed $x$ this term is bounded as $t\to\infty$, because the exponential
factor decays. Along $x=t+\xi$, however, it is proportional to
$\eps t\,\xi^3e^{-3\xi^2}$ and grows linearly in $t$ whenever $\xi\ne0$.
Thus the regular expansion loses its ordering on the slow scale
$t=O(\eps^{-1})$. Boundedness at each fixed spatial point is therefore not a
uniform validity test for a traveling wave.

The standard remedy is to absorb the accumulated change into a leading profile
or parameter that varies on a slow scale. For the nonlinear wave family
$U_{tt}=(1+\eps U_x^2)U_{xx}$ and related equations studied in
\cite{Tarayrah2023,McAdam2025}, the slow time $T=\eps t$ gives
evolution equations for the leading left- and right-moving profiles. For the
periodic orbit below, the same idea takes the Poincar\'e--Lindstedt form
\cite{Nayfeh1973}: the growing phase terms are absorbed into a corrected
wavenumber.

\subsection{The reduced equation and its periodic orbit}

Let $U=U(\xi)$, $\xi=x-ct$, and $\kappa=c^2-1\ne0$. The periodic orbits
used below require $\kappa\lambda>0$. Consider the orbit through the
initial data $U(0)=\mathcal A$, $U'(0)=0$.
The initial point is chosen at a positive turning point. The formulas below
are given for positive $\kappa$ and $\lambda$; when both are negative,
multiplication of the reduced equation by $-1$ gives the same form with
$(\kappa,\lambda,F,G)$ replaced by $(-\kappa,-\lambda,-F,-G)$. The
assumptions used below are
\begin{equation}\label{eq:twhyp}
\mathcal A>0,\qquad \kappa>0,\qquad \lambda>0,
\end{equation}
with $F$ and $G$ sufficiently smooth on a neighbourhood of
$[-\mathcal A,\mathcal A]$. In this subsection primes denote $d/d\xi$.
Equation \eqref{eq:wavefull} becomes
\begin{equation}\label{eq:twode}
\kappa U''+\lambda U^3+\eps F(U)+\eps^2G(U)=0,
\end{equation}
with first integral
\begin{equation}\label{eq:twfirstint}
\frac{\kappa}{2}U'^2+V_\eps(U)=\mathcal H,
\qquad
V_\eps(u)=\frac{\lambda}{4}u^4+\eps f(u)+\eps^2g(u),
\qquad F=f',\quad G=g'.
\end{equation}
With the initial data $U(0)=\mathcal A$, $U'(0)=0$, one has
$\mathcal H=V_\eps(\mathcal A)$. For sufficiently small $|\eps|$, the lower
turning point adjacent to $\mathcal A$ persists; denote it by
$U_-(\eps)<\mathcal A$. The periodic orbit then has wavelength
\begin{equation}\label{eq:twwavelength}
X(\mathcal A,\eps)
=\sqrt{2\kappa}\int_{U_-(\eps)}^{\mathcal A}
\frac{du}{\sqrt{V_\eps(\mathcal A)-V_\eps(u)}}.
\end{equation}
This quadrature defines the orbit locally for general smooth perturbations. If
the total potential $V_\eps$ is a polynomial of degree at most four, the
quadrature is elementary or elliptic; for more general potentials it is
generically hyperelliptic or nonclassical, although special further reductions
may occur.

Substitution of
$U=U_0+\eps U_1+\eps^2U_2+o(\eps^2)$ into \eqref{eq:twode} gives
\begin{subequations}\label{eq:twsys}
\begin{align}
&\kappa U_0''+\lambda U_0^3=0,\label{eq:twsys0}\\
&\kappa U_1''+3\lambda U_0^2U_1=-F(U_0),\label{eq:twsys1}\\
&\kappa U_2''+3\lambda U_0^2U_2
=-3\lambda U_0U_1^2-F'(U_0)U_1-G(U_0).\label{eq:twsys2}
\end{align}
\end{subequations}
These equations are precisely the traveling-wave reduction of the
second-order FS system \eqref{eq:E0}--\eqref{eq:E2}; in particular, the same
linearized operator acts on each newly introduced coefficient. The example
therefore tests not the formal correctness of that hierarchy, but the
long-scale validity of its regular partial sums.

The solution of \eqref{eq:twsys0} with
$U_0(0)=\mathcal A$, $U_0'(0)=0$ is
\begin{equation}\label{eq:twU0}
U_0=\mathcal A\operatorname{cn}\!\left(\omega_0\xi,\tfrac1{\sqrt2}\right),
\qquad
\omega_0=\mathcal A\sqrt{\lambda/\kappa}.
\end{equation}
Set
\begin{equation}\label{eq:twtheta}
\theta=\omega_0\xi,
\qquad K=K\!\left(\tfrac1{\sqrt2}\right),
\qquad
\langle h\rangle=\frac1{4K}\int_0^{4K}h(\theta)\,d\theta.
\end{equation}
From here on primes denote $d/d\theta$. The period is $4K$, and
\begin{equation}\label{eq:twU0derivs}
\begin{aligned}
U_0(\theta)&=\mathcal A\operatorname{cn}\!\left(\theta,\tfrac1{\sqrt2}\right),\\
U_0'(\theta)&=-\mathcal A\operatorname{sn}\!\left(\theta,\tfrac1{\sqrt2}\right)
\operatorname{dn}\!\left(\theta,\tfrac1{\sqrt2}\right),\\
U_0''(\theta)&=-\frac{U_0(\theta)^3}{\mathcal A^2}.
\end{aligned}
\end{equation}
Using $\kappa\omega_0^2=\lambda\mathcal A^2$, the correction equations become
\begin{equation}\label{eq:twcorr}
\begin{aligned}
U_1''+\frac{3U_0^2}{\mathcal A^2}U_1
&=-\frac{F(U_0)}{\lambda\mathcal A^2},\\
U_2''+\frac{3U_0^2}{\mathcal A^2}U_2
&=-\frac{3\lambda U_0U_1^2+F'(U_0)U_1+G(U_0)}
{\lambda\mathcal A^2}.
\end{aligned}
\end{equation}
The homogeneous equation has the fundamental solutions
\begin{equation}\label{eq:twhom}
y_1(\theta)=U_0'(\theta),
\qquad
y_2(\theta)=U_0(\theta)+\theta U_0'(\theta),
\qquad
y_1y_2'-y_1'y_2=\mathcal A^2.
\end{equation}
The second solution is the characteristic of the dilation \eqref{eq:D0}
evaluated on the traveling wave.

\subsection{Secular corrections and phase renormalization}

\begin{proposition}[First correction]\label{prop:twU1}
The solution of the first equation in \eqref{eq:twcorr} with
$U_1(0)=U_1'(0)=0$ is
\begin{equation}\label{eq:twU1}
U_1(\theta)=\frac1{\lambda\mathcal A^4}
\Bigl[U_0'(\theta)\bigl(\Phi(\theta)+f(\mathcal A)\theta\bigr)
-U_0(\theta)\bigl(f(U_0(\theta))-f(\mathcal A)\bigr)\Bigr],
\end{equation}
where
\begin{equation}\label{eq:twPhi}
\Phi(\theta)=\int_0^\theta
\left[U_0(\tau)F(U_0(\tau))-f(U_0(\tau))\right]d\tau.
\end{equation}
It decomposes as
\begin{equation}\label{eq:twS1}
U_1(\theta)=S_1\theta U_0'(\theta)+P_1(\theta),
\qquad
S_1=\frac1{\lambda\mathcal A^4}
\left[\frac{\Phi(4K)}{4K}+f(\mathcal A)\right],
\end{equation}
where
\begin{equation}\label{eq:twP1}
\begin{aligned}
P_1(\theta)&=\frac1{\lambda\mathcal A^4}
\Bigl[U_0'(\theta)\widetilde\Phi(\theta)
-U_0(\theta)\bigl(f(U_0(\theta))-f(\mathcal A)\bigr)\Bigr],\\
\widetilde\Phi(\theta)&=\Phi(\theta)-\frac{\Phi(4K)}{4K}\theta.
\end{aligned}
\end{equation}
The function $P_1$ is $4K$-periodic and satisfies
$P_1(0)=P_1'(0)=0$.
\end{proposition}

\begin{proof}
Since $U_0'F(U_0)=(f(U_0))'$ and
\[
\int_0^\theta\bigl(U_0+\tau U_0'\bigr)F(U_0)\,d\tau
=\Phi(\theta)+\theta f(U_0(\theta)),
\]
variation of parameters with \eqref{eq:twhom} gives \eqref{eq:twU1}. The
integrand in \eqref{eq:twPhi} is periodic, so
$\Phi(\theta)=(\Phi(4K)/(4K))\theta+\widetilde\Phi(\theta)$ with
$\widetilde\Phi$ periodic. This yields \eqref{eq:twS1}--\eqref{eq:twP1}.
\end{proof}

\begin{proposition}[Second correction and corrected wavenumber]
\label{prop:twU2}
For the periodic orbit described above, define
\begin{equation}\label{eq:twomega}
\omega(\mathcal A,\eps):=\frac{4K}{X(\mathcal A,\eps)}
=\omega_0\bigl(1+\eps S_1+\eps^2S_2+o(\eps^2)\bigr).
\end{equation}
The first coefficient agrees with \eqref{eq:twS1}. The solution of the second
equation in \eqref{eq:twcorr} with $U_2(0)=U_2'(0)=0$ has the decomposition
\begin{equation}\label{eq:twU2}
U_2(\theta)=\frac{S_1^2}{2}\theta^2U_0''(\theta)
+S_2\theta U_0'(\theta)
+S_1\theta P_1'(\theta)+P_2(\theta),
\end{equation}
where $P_2$ is $4K$-periodic and
$P_2(0)=P_2'(0)=0$. Besides differentiation of the period
\eqref{eq:twwavelength}, $S_2$ may be obtained from the one-period monodromy
condition
\begin{equation}\label{eq:twS2mono}
S_2=\frac{U_2'(4K)}{4K\,U_0''(0)}-S_1^2
-\frac{S_1P_1''(0)}{U_0''(0)}.
\end{equation}
Thus $P_2$ is computed from one linear initial-value problem and the periodicity
condition, even when no elementary expression for the exact orbit is available.
\end{proposition}

\begin{proof}
The periodic orbit depends smoothly on $\eps$ while the two turning points
remain simple. Expanding it at fixed $\xi$ gives
$U_0+\eps U_1+\eps^2U_2$. Expanding it instead at the fixed phase, write the
orbit as $Q(\phi,\eps)=U_0(\phi)+\eps P_1(\phi)+\eps^2P_2(\phi)+o(\eps^2)$
with $\phi=\bigl(\omega(\mathcal A,\eps)/\omega_0\bigr)\theta$ and all
coefficient functions $4K$-periodic; then, at fixed $\theta$,
\[
U_0(\phi)=U_0(\theta)+(\eps S_1+\eps^2S_2)\,\theta U_0'(\theta)
+\tfrac{\eps^2S_1^2}{2}\,\theta^2U_0''(\theta)+o(\eps^2),
\qquad
P_1(\phi)=P_1(\theta)+\eps S_1\theta P_1'(\theta)+o(\eps),
\]
and comparing the two expansions of the same orbit yields
\eqref{eq:twS1} and \eqref{eq:twU2}. Differentiating \eqref{eq:twU2} and
setting $\theta=4K$, where
$U_0'=U_0'''=P_1'=P_2'=0$ while $U_0''(4K)=U_0''(0)$ and
$P_1''(4K)=P_1''(0)$ by periodicity, gives \eqref{eq:twS2mono}.
\end{proof}

The order-consistent phase approximations use different truncations of the
wavenumber:
\begin{equation}\label{eq:twomegatrunc}
\omega_{[1]}=\omega_0(1+\eps S_1),
\qquad
\omega_{[2]}=\omega_0(1+\eps S_1+\eps^2S_2).
\end{equation}
They are
\begin{equation}\label{eq:twren}
\begin{aligned}
U^{[1]}(\xi)&=U_0(\omega_{[1]}\xi)
+\eps P_1(\omega_{[1]}\xi),\\
U^{[2]}(\xi)&=U_0(\omega_{[2]}\xi)
+\eps P_1(\omega_{[2]}\xi)
+\eps^2P_2(\omega_{[2]}\xi).
\end{aligned}
\end{equation}
Expansion at fixed $\xi$ reproduces the regular FS series through first and
second order, respectively. Under the stated smoothness assumptions, their
residuals are $O(\eps^2)$ and $O(\eps^3)$, and all displayed profile functions
remain bounded in $\xi$. An omitted higher-order wavenumber correction can
still accumulate over a sufficiently long interval; the purpose of
\eqref{eq:twomegatrunc} is to compare the two asymptotic orders without placing
the second-order phase information into the first-order approximation.

\subsection{Integer-power family and numerical validation}

Let $F(U)=AU^N$ and $G(U)=BU^M$, with integers $N,M\ge2$. Then
$f(u)=Au^{N+1}/(N+1)$, $g(u)=Bu^{M+1}/(M+1)$, and
\begin{equation}\label{eq:twS1pow}
S_1=\frac{A\,(NI_{N+1}+1)}{\lambda(N+1)}\mathcal A^{N-3},
\qquad
I_m:=\left\langle\operatorname{cn}^m\!\left(\theta,\tfrac1{\sqrt2}\right)
\right\rangle.
\end{equation}
For nonnegative integer $m$, one has $I_m=0$ when $m$ is odd.
Differentiation of the product
$\operatorname{sn}\!\left(\theta,\tfrac1{\sqrt2}\right)
\operatorname{dn}\!\left(\theta,\tfrac1{\sqrt2}\right)
\operatorname{cn}^{m+1}\!\left(\theta,\tfrac1{\sqrt2}\right)$ gives
\begin{equation}\label{eq:twrecur}
\begin{aligned}
\int_0^\theta\operatorname{cn}^{m+4}\!\left(\tau,\tfrac1{\sqrt2}\right)d\tau
&=\frac{m+1}{m+3}
\int_0^\theta\operatorname{cn}^{m}\!\left(\tau,\tfrac1{\sqrt2}\right)d\tau\\
&\quad+\frac{2}{m+3}\,
\operatorname{sn}\!\left(\theta,\tfrac1{\sqrt2}\right)
\operatorname{dn}\!\left(\theta,\tfrac1{\sqrt2}\right)
\operatorname{cn}^{m+1}\!\left(\theta,\tfrac1{\sqrt2}\right),
\qquad
I_{m+4}=\frac{m+1}{m+3}I_m.
\end{aligned}
\end{equation}
The base integrals are
\begin{equation}\label{eq:twbases}
\begin{aligned}
&\int_0^\theta d\tau=\theta,
\qquad
\int_0^\theta\operatorname{cn}\!\left(\tau,\tfrac1{\sqrt2}\right)d\tau
=\sqrt2\,\arcsin\!\left[
\tfrac1{\sqrt2}\operatorname{sn}\!\left(\theta,\tfrac1{\sqrt2}\right)
\right],\\
&\int_0^\theta\operatorname{cn}^2\!\left(\tau,\tfrac1{\sqrt2}\right)d\tau
=2\cE(\theta)-\theta,
\qquad
\cE(\theta)=\int_0^\theta\operatorname{dn}^2
\!\left(\tau,\tfrac1{\sqrt2}\right)d\tau,\\
&\int_0^\theta\operatorname{cn}^3\!\left(\tau,\tfrac1{\sqrt2}\right)d\tau
=\operatorname{sn}\!\left(\theta,\tfrac1{\sqrt2}\right)
\operatorname{dn}\!\left(\theta,\tfrac1{\sqrt2}\right).
\end{aligned}
\end{equation}
Thus $I_4=1/3$, $I_8=5/21$, and
\[
I_2=\frac{2E(1/\sqrt2)}{K}-1
=\frac{8\pi^2}{\Gamma(1/4)^4}.
\]
The periodic part of the first correction is
\begin{equation}\label{eq:twP1pow}
P_1(\theta)=\frac{A\mathcal A^{N-2}}{\lambda(N+1)}
\left[\operatorname{cn}\!\left(\theta,\tfrac1{\sqrt2}\right)
-\operatorname{cn}^{N+2}\!\left(\theta,\tfrac1{\sqrt2}\right)
-N\operatorname{sn}\!\left(\theta,\tfrac1{\sqrt2}\right)
\operatorname{dn}\!\left(\theta,\tfrac1{\sqrt2}\right)
\widetilde C_N(\theta)\right],
\end{equation}
where
\begin{equation}\label{eq:twCtilde}
\widetilde C_N(\theta)
=\int_0^\theta\operatorname{cn}^{N+1}\!\left(\tau,\tfrac1{\sqrt2}\right)d\tau
-I_{N+1}\theta.
\end{equation}
In $S_2$, the contribution of $g$ is linear and is obtained from
\eqref{eq:twS1pow} by replacing $(A,N)$ with $(B,M)$; the remaining contribution
is quadratic in $A$. For the integer-power members of the power family
\eqref{eq:classgeneric}, one has $N\ge4$ and $M=2N-3$. Both contributions to $S_2$ then scale as
$\mathcal A^{2(N-3)}$, so the expansion of $\omega/\omega_0$ depends on
$(\eps,\mathcal A)$ through the combination
$\eps\mathcal A^{N-3}$ of Remark~\ref{rem:weights}.

For the numerical comparison, take
$\lambda=1$, $c=\sqrt2$, $\kappa=1$, the wave amplitude $\mathcal A=0.9$,
the perturbation coefficients $A=B=1$, and the exponents $N=5$ and
$M=7$. (Throughout this section, calligraphic $\mathcal A$ denotes the wave
amplitude, while $A$ and $B$ denote the coefficients of $F$ and $G$, as in
Theorem~\ref{thm:classification}.) Then
\begin{equation}\label{eq:twnumbers}
S_1=0.3200633653,
\qquad
S_2=0.1638741121.
\end{equation}
The second value is obtained from \eqref{eq:twS2mono}. Independently, since
$N$ and $M$ are odd, the potential $V_\eps$ is even, the turning points are
$-\mathcal A$ and $\mathcal A$, and the expansion of the period integral
\eqref{eq:twwavelength} in $\eps$ can be evaluated to machine precision by a
quadrature with the turning-point singularities absorbed; this gives
$S_2=0.1638741121$, in agreement with \eqref{eq:twS2mono} to better
than $10^{-10}$. For $\eps=0.1$, $0.25$, and $0.5$, the comparison window is
\[
0\le\xi\le\Xi=\frac{4.5}{\eps S_1\omega_0},
\]
so $\Xi=156.2$, $62.5$, and $31.2$. Using the exact perturbed wavelength, these
windows contain approximately $19.59$, $8.26$, and $4.53$ wavelengths. The
case $\eps=0.5$ is included as an additional comparison at a moderately large
perturbation value. It shows how the phase-renormalized approximations behave
beyond the small-parameter range and is not used to infer an asymptotic error
law as $\eps\to0$.

Table~\ref{tab:twerrors} reports the maximal absolute errors on the initial
and final $10\%$ of each comparison window. Since $\Xi$ scales as
$\eps^{-1}$, these windows have different lengths; their initial $10\%$
contains approximately $2.0$, $0.8$, and $0.45$ wavelengths, respectively.
The windows are chosen to give the same accumulated first-order phase
correction, so comparisons between the rows are not intended as a convergence
study in $\eps$. Because both $\eps$ and the interval length vary, the error
of the regular partial sum on the initial segments is not expected to be
monotone in $\eps$.

Figures~\ref{fig:cn010}--\ref{fig:cn050} show the last three wavelengths and the error over
the full window. The regular partial sum becomes dominated by its secular
terms. The first-order approximation remains bounded, but the omitted
$\eps^2S_2$ phase correction accumulates over the $O(\eps^{-1})$ window. The
second-order approximation substantially reduces this accumulated phase error: over the final $10\%$ of the
three windows, its maximal error is smaller than that of $U^{[1]}$ by factors
of approximately $27$, $11$, and $5$.

\begin{table}[H]
\centering
\caption{Maximal absolute errors over the initial and final $10\%$ of the
comparison window. The first data column in each group is the regular partial
sum $U_0+\eps U_1+\eps^2U_2$; the table is generated by the script
described at the end of this section.}
\label{tab:twerrors}
\small

% Generated by tw_compare.m; do not edit by hand.
\begin{tabular}{lcccccc}
\toprule
& \multicolumn{3}{c}{$0\leq\xi\leq\Xi/10$} & \multicolumn{3}{c}{$0.9\,\Xi\leq\xi\leq\Xi$}\\
\cmidrule(lr){2-4}\cmidrule(lr){5-7}
$\eps$ & $U_0{+}\eps U_1{+}\eps^2U_2$ & $U^{[1]}$ & $U^{[2]}$ & $U_0{+}\eps U_1{+}\eps^2U_2$ & $U^{[1]}$ & $U^{[2]}$\\
\midrule
$0.10$ & $1.94\cdot10^{-2}$ & $1.29\cdot10^{-2}$ & $4.63\cdot10^{-4}$ & $7.67$ & $1.38\cdot10^{-1}$ & $5.06\cdot10^{-3}$\\
$0.25$ & $1.18\cdot10^{-2}$ & $3.25\cdot10^{-2}$ & $2.90\cdot10^{-3}$ & $8.26$ & $3.31\cdot10^{-1}$ & $3.11\cdot10^{-2}$\\
$0.50$ & $3.79\cdot10^{-2}$ & $4.58\cdot10^{-2}$ & $7.24\cdot10^{-3}$ & $7.43$ & $5.71\cdot10^{-1}$ & $1.10\cdot10^{-1}$\\
\bottomrule
\end{tabular}

\end{table}

\begin{figure}[t]
\centering
\includegraphics[width=0.92\textwidth]{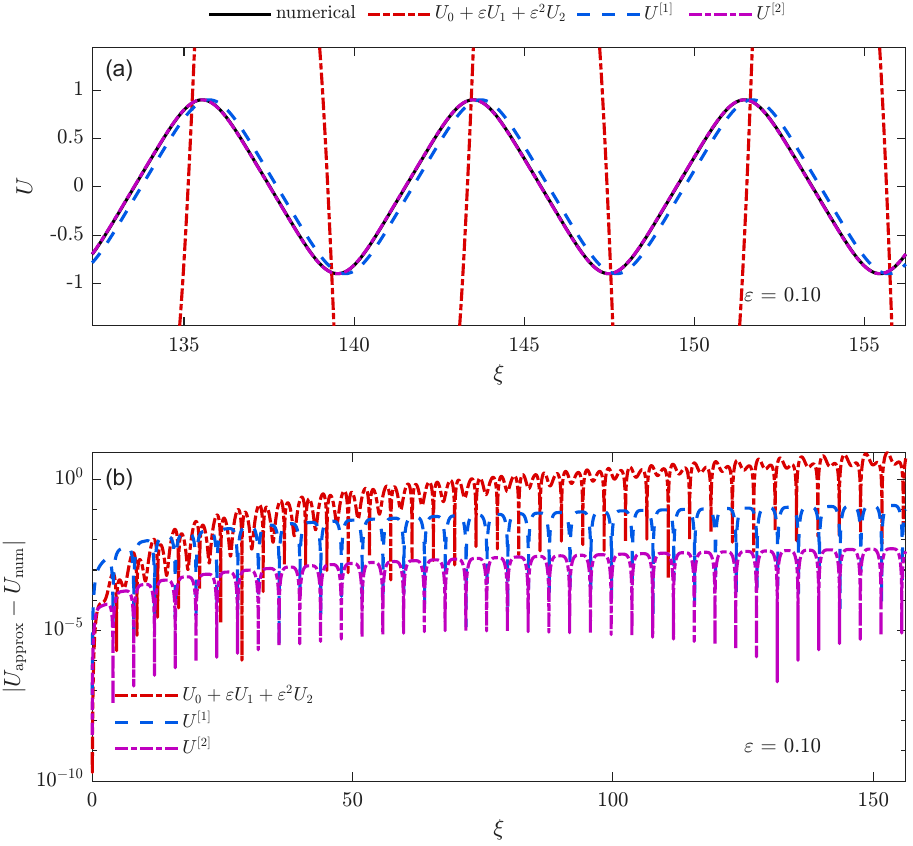}
\caption{$\eps=0.1$, window $0\leq\xi\leq156.2$. (a)~The last three exact
wavelengths: the numerical solution of \eqref{eq:twode}, the regular FS
partial sum $U_0+\eps U_1+\eps^2U_2$, and the order-consistent approximations
$U^{[1]}$ and $U^{[2]}$ of \eqref{eq:twren}. The regular partial sum leaves
the displayed vertical range over part of the interval; its error remains
fully shown in panel~(b). (b)~Absolute errors of the three approximations over
the full window. The second-order approximation includes the $O(\eps^2)$ wavenumber
correction and substantially reduces the accumulated phase error.}
\label{fig:cn010}
\end{figure}

\begin{figure}[t]
\centering
\includegraphics[width=0.92\textwidth]{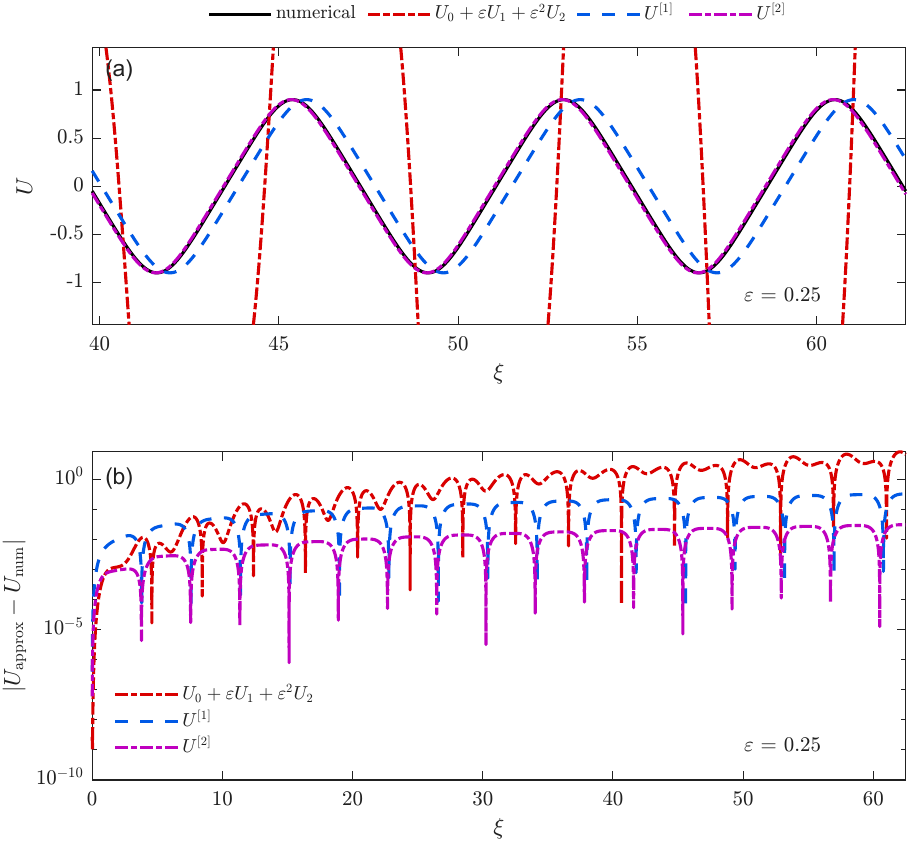}
\caption{$\eps=0.25$, window $0\leq\xi\leq62.5$. (a)~The last three exact
wavelengths: the numerical solution of \eqref{eq:twode}, the regular FS
partial sum $U_0+\eps U_1+\eps^2U_2$, and the order-consistent approximations
$U^{[1]}$ and $U^{[2]}$ of \eqref{eq:twren}. The regular partial sum leaves
the displayed vertical range over part of the interval; its error remains
fully shown in panel~(b). (b)~Absolute errors of the three approximations over
the full window. The second-order approximation includes the $O(\eps^2)$ wavenumber
correction and substantially reduces the accumulated phase error.}
\label{fig:cn025}
\end{figure}

\begin{figure}[t]
\centering
\includegraphics[width=0.92\textwidth]{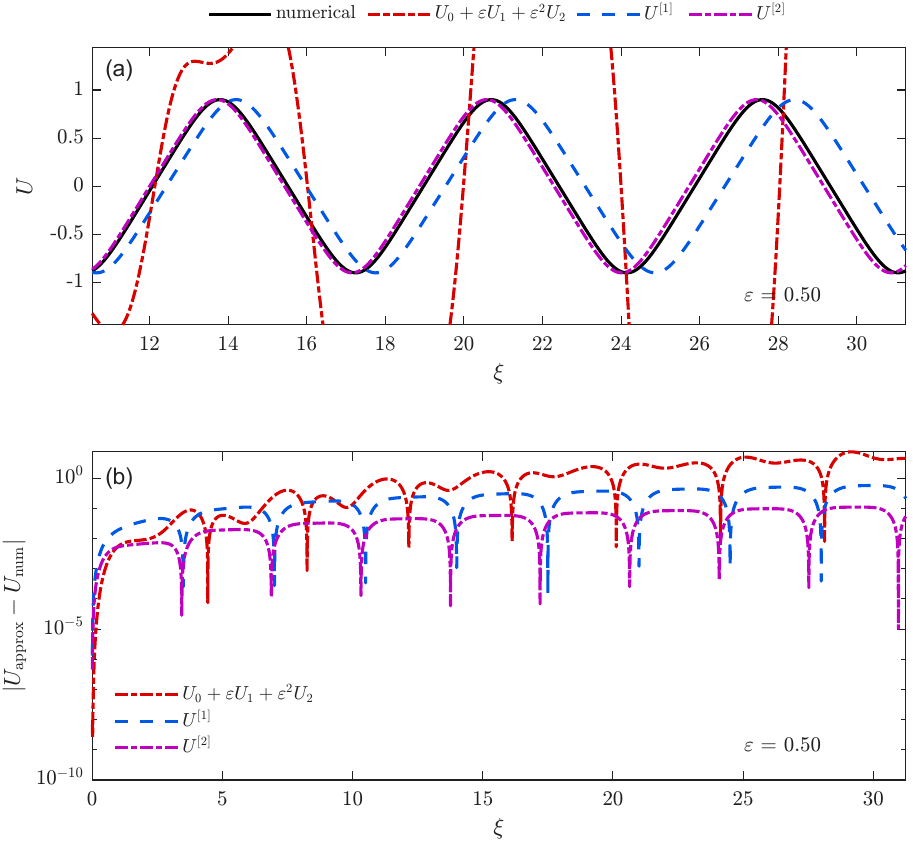}
\caption{$\eps=0.5$, window $0\leq\xi\leq31.2$. (a)~The last three exact
wavelengths: the numerical solution of \eqref{eq:twode}, the regular FS
partial sum $U_0+\eps U_1+\eps^2U_2$, and the order-consistent approximations
$U^{[1]}$ and $U^{[2]}$ of \eqref{eq:twren}. The regular partial sum leaves
the displayed vertical range over part of the interval; its error remains
fully shown in panel~(b). (b)~Absolute errors of the three approximations over
the full window. The second-order approximation includes the $O(\eps^2)$ wavenumber
correction and substantially reduces the accumulated phase error.}
\label{fig:cn050}
\end{figure}

\FloatBarrier
\begin{remark}[Fourier structure of the error]\label{rem:twfourier}
The comparison above suggests that the error of the renormalized
approximations \eqref{eq:twren} is dominated by a residual wavenumber
mismatch rather than by a deformation of the wave. To separate the two
effects, define, for $j=1,2$,
\begin{equation}\label{eq:twmeasures}
E_{\mathrm{pos}}^{[j]}
=\frac{\bigl\|U^{[j]}-U\bigr\|_{L^2(0,\Xi)}}{\|U\|_{L^2(0,\Xi)}},
\qquad
E_{\omega}^{[j]}
=\frac{\bigl|\omega_{[j]}-\omega\bigr|}{\omega},
\qquad
E_{\mathrm{shape}}^{[j]}
=\Bigl[\sum_{n=1}^{10}
\bigl(\hat v^{[j]}_n-\hat v_n\bigr)^2\Bigr]^{1/2},
\end{equation}
where $U$ is the numerical solution of \eqref{eq:twode},
$\omega=\omega(\mathcal A,\eps)=4K/X(\mathcal A,\eps)$ is its wavenumber
\eqref{eq:twomega}, $\omega_{[j]}$ are the truncated wavenumbers
\eqref{eq:twomegatrunc}, and, for a periodic function $V$ with period $X_V$,
\begin{equation}\label{eq:twharm}
v_n=\left|\frac{1}{X_V}\int_0^{X_V}V(\xi)\,e^{-2\pi i n\xi/X_V}\,d\xi\right|,
\qquad
\hat v_n=v_n\Bigl(\sum_{k=1}^{10}v_k^2\Bigr)^{-1/2},
\end{equation}
computed with the period $X(\mathcal A,\eps)$ for $U$ and $4K/\omega_{[j]}$
for $U^{[j]}$. Thus $E_{\mathrm{pos}}$ is the relative solution error on the
comparison window; $E_{\omega}$ is the relative wavenumber error; and
$E_{\mathrm{shape}}$ compares the normalized magnitudes of the first ten
Fourier modes of one period of each wave. The last diagnostic is insensitive
to phase translation, period rescaling, and overall amplitude scaling, and
thus isolates a deformation of the waveform from a wavenumber mismatch; it
is not a complete metric on waveform shapes, since distinct profiles may
share a finite set of spectral magnitudes.
Figure~\ref{fig:cnfour} shows the three measures for $0.05\le\eps\le0.5$,
together with the values $S_2\eps^2$ and $|S_3|\eps^3$, where
$S_3=-0.0586$ is the third-order coefficient of $\omega/\omega_0$, obtained
by carrying the $\eps$-expansion of the period integral
\eqref{eq:twwavelength} one order further. The wavenumber errors follow these
values, and the waveform errors lie roughly two to three and a half orders
of magnitude below the position errors, between $1.5\cdot10^{-6}$ and
$1.3\cdot10^{-3}$ across the range: the error of the renormalized
approximations is almost entirely a wavenumber error. The regular partial sum enters through its
position error, which exceeds unity on these windows.
\end{remark}

\begin{figure}[t]
\centering
\includegraphics[width=0.9\textwidth]{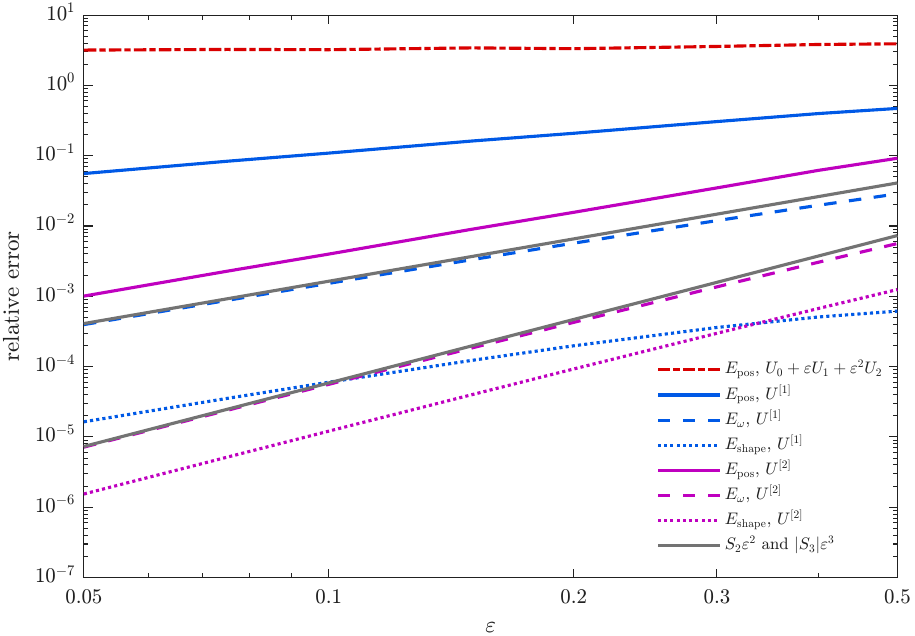}
\caption{The measures \eqref{eq:twmeasures} as functions of $\eps$: position
errors (solid), wavenumber errors (dashed), and waveform errors (dotted) for
$U^{[1]}$ and $U^{[2]}$, the position error of the regular partial sum
(dash-dotted), and the values $S_2\eps^2$ and $|S_3|\eps^3$ (thin gray).}
\label{fig:cnfour}
\end{figure}

\begin{remark}[Numerical verification]\label{rem:twverify}
The script described below evaluates $S_2$ independently
from the monodromy condition \eqref{eq:twS2mono} and from the exact-order
expansion of the regularized period integral; the two values agree to better
than $10^{-10}$. Over one period, the endpoint defects in $P_1$, $P_1'$,
$P_2$, and $P_2'$ are all below $1.6\cdot10^{-11}$, the defect in $P_2'$
being at roundoff level. The quadrature formula
\eqref{eq:twU1} agrees with direct integration of the first correction
equation to $4.8\cdot10^{-9}$ over three periods. The tabulated error maxima
are stable to a relative $3.1\cdot10^{-7}$ under halving of the evaluation
grid. The script generates Table~\ref{tab:twerrors},
Figures~\ref{fig:cn010}--\ref{fig:cnfour}, and the output files described below.
\end{remark}

\subsection*{Code and data availability}

The numerical results of this section are produced by a single
MATLAB/GNU Octave script, \texttt{tw\_compare.m}. The script and its
output files are included as ancillary code and data with the arXiv submission.
The numerical procedure is specified as follows. The correction equations
\eqref{eq:twcorr} are integrated as a first-order system in $\theta$ with the
explicit Runge--Kutta solver \texttt{ode45}, relative tolerance
$2\cdot10^{-12}$, absolute tolerance $2\cdot10^{-13}$, and maximum step
$4K/250$, on a grid of $500$ points per period extending over the longest
comparison window; the periodic profiles $P_1$ and $P_2$ are then extracted
on $6001$ points of one period by cubic spline interpolation and the
subtraction of the secular terms in \eqref{eq:twS1} and \eqref{eq:twU2}, with
$S_2$ taken from the monodromy condition \eqref{eq:twS2mono}. The exact
perturbed orbit is integrated by the same solver and tolerances on the
evaluation grid in $\xi$. Period integrals are evaluated by Gauss--Chebyshev
quadrature after the affine substitution $u=u_c+u_r\,s$, $s\in[-1,1]$, which
absorbs both turning-point square-root singularities; the exact perturbed
wavelength \eqref{eq:twwavelength} uses $2400$ nodes with the lower turning
point located by bracketing and \texttt{fzero}, whereas the coefficients
$S_1$, $S_2$, and $S_3$ of the expansion of $\omega/\omega_0$ are obtained with
$400$ nodes from the exact-order expansion of the period integrand, using the
evenness of $V_\eps$ for odd $N$ and $M$, so that the turning points are
$-\mathcal A$ and $\mathcal A$ and no numerical differencing is involved. The error maxima of
Table~\ref{tab:twerrors} use $1200$ evaluation points per exact wavelength,
with the discrete maximum refined by a three-point parabola, and are checked
against $600$ points per wavelength. The Fourier diagnostic
\eqref{eq:twmeasures} uses $1024$ samples of one period starting at
$\xi=0.55\,\Xi$. The script writes the table, the curve data of
Figures~\ref{fig:cn010}--\ref{fig:cn050}, the data of Figure~\ref{fig:cnfour},
and a verification log containing the quantities quoted in
Remark~\ref{rem:twverify}.

\FloatBarrier
\section{Conclusions and outlook}\label{sec:discussion}

The Fushchych--Shtelen construction has been formulated at arbitrary
perturbation order. The coefficient equations \eqref{eq:FSsystemGeneral}
have the repeated linear structure displayed in \eqref{eq:triangular}:
each new perturbation level is governed by the same linearized operator,
with its source determined by the lower-order coefficients.
Theorem~\ref{thm:correspondence} identifies, under the stated hypotheses,
which FS symmetries arise from a BGI approximate characteristic. The two
principal definitions of approximate symmetry are thereby related while
their different admissibility conditions are retained.

For the perturbed cubic wave equation, Theorem~\ref{thm:classification}
classifies the pairs admitting an FS point symmetry with prescribed
projection onto the unperturbed dilation, modulo the formal equivalences of
Section~\ref{sec:wave}. Perturbations nontrivial at first order give the power
family \eqref{eq:classgeneric} and three exceptional logarithmic families;
Proposition~\ref{prop:degenerateFS} gives the additional families for
perturbations first appearing at order $\eps^2$.
Proposition~\ref{prop:BGIraw} classifies all pairs admitting a second-order
BGI point symmetry continuation of the same dilation, with corrections
allowed in both the independent and dependent variables. Among the families
nontrivial at first order, the common pair is \eqref{eq:BGInormalFG}, with
the additional coefficient condition \eqref{eq:BGIlockB}.
Corollary~\ref{cor:BGIdegenerate} gives the common family when the first
nontrivial perturbation occurs at second order. The BGI-admissible pairs
form a proper subset of those in the FS dilation classification. Their
general BGI symmetry characteristics have FS counterparts under
Theorem~\ref{thm:correspondence}; the continuations
\eqref{eq:BGIverticalrep} give the point symmetries displayed in
Section~\ref{sec:BGIwave}.

The classification is used to construct explicit solutions. Under the
Lorentz--dilation reduction, the leading field is given by
\eqref{eq:Wsol} and satisfies the d'Alembert--Hamilton relation
\eqref{eq:dHam}. The correction coefficients follow by integration of the
reduced invariance equations and are given in \eqref{eq:generictower} and
\eqref{eq:towerfivehalf}--\eqref{eq:towerquadraticlog}. For the common
family \eqref{eq:BGInormalFG}, the solution \eqref{eq:intersectionsol}
is the second-order expansion \eqref{eq:anomalous} of a power with a
perturbation-dependent exponent. In this family, the logarithmic corrections
describe the change in the scaling exponent.

For traveling cnoidal waves, the first correction is obtained explicitly by
quadrature, and all second-order secular terms are identified in
\eqref{eq:twU2}. The periodic second-order remainder is computed from a
linear initial-value problem together with the monodromy condition
\eqref{eq:twS2mono}. Using the order-consistent wavenumbers
\eqref{eq:twomegatrunc}, the approximations \eqref{eq:twren} remain bounded,
and the second-order approximation includes the $O(\eps^2)$ wavenumber
correction omitted from the first-order approximation
(Table~\ref{tab:twerrors} and Figures~\ref{fig:cn010}--\ref{fig:cn050}).
For the integer-power members of the power family \eqref{eq:classgeneric},
the wavenumber expansion involves the amplitude and perturbation parameter
through $\eps\mathcal A^{N-3}$. The scaling weights of the classification
therefore also occur in the traveling-wave reduction, although that
reduction alone does not determine the classification.

Approximate symmetry provides both a classification of the symmetries of a
limiting equation that can be continued under perturbation and a means of
constructing approximate solutions. The latter requires a separate check of
the scale on which the expansion remains ordered. The examples in
Section~\ref{sec:secular} and the cnoidal-wave calculation show how secular
growth can invalidate a formally correct FS partial sum on long scales.
Phase renormalization removes the computed secular terms and extends the
useful scale of the approximation; higher-order phase corrections may still
accumulate on longer scales.

Further directions include the second-order classification of continuations
of the special conformal generators of the cubic wave equation, an
all-orders treatment of equations with differential identities, and a
corresponding higher-order relation between BGI approximate conservation
laws \cite{CTY2025} and exact conservation laws of FS systems. These
extensions would relate further approximate structures of the original
equations to exact structures of their coefficient systems, with the
approximation order and its domain of validity stated explicitly.

\section*{Conflict of Interest}
The author declares that there is no conflict of interest concerning the
publication of this article.

\section*{Funding}
The author acknowledges support from the Natural Sciences and Engineering
Research Council of Canada through Discovery Grant RGPIN-2024-04308.

\bibliographystyle{ieeetr}
\bibliography{UMJ_Fushchych_v20}

\end{document}